\documentclass[reqno]{amsart}
\usepackage{txfonts}
\usepackage{amsfonts}
\usepackage{amssymb}
\usepackage[mathscr]{eucal}
\usepackage{amsmath}
\usepackage[pdfstartview=FitH,
            CJKbookmarks=true,
            bookmarksnumbered=true,
            bookmarksopen=true,
            colorlinks,
            citecolor=blue,linkcolor=blue]{hyperref}
\usepackage{ytableau}
\usepackage{youngtab}
\usepackage{float}
\usepackage{url}
\usepackage{appendix}
\usepackage{multirow}
\usepackage{enumitem}

\newtheorem{theorem}{Theorem}[section]
\newtheorem{corollary}[theorem]{Corollary}
\newtheorem{lemma}[theorem]{Lemma}
\newtheorem{proposition}[theorem]{Proposition}
\theoremstyle{definition}
\newtheorem{definition}[theorem]{Definition}

\newtheorem{example}[theorem]{Example}
\newtheorem{problem}[theorem]{Problem}

\newtheorem{conjecture}{Conjecture}
\numberwithin{equation}{section}
\theoremstyle{remark}

\allowdisplaybreaks[4]
\newcommand{\tabincell}[2]{\begin{tabular}{@{}#1@{}}#2\end{tabular}}

\begin{document}

\title{Saxl Conjecture for triple hooks}
\pdfbookmark[0]{Saxl Conjecture for triple hooks}{}

\author{Xin Li}
\address{Department of Mathematics, Zhejiang University of Technology, Hangzhou 310023, P. R. China}
\email{xinli1019@126.com}
\thanks{The research is supported by National Natural Science Foundation of China (Grant No.11801506).}

\subjclass[2010]{Primary 20C30; Secondary 05E15}
\keywords{Saxl Conjecture, Dominance order,  Kronecker coefficient, Semigroup property, Tensor square}

\begin{abstract}
We make some progresses on Saxl Conjecture. Firstly, we show that the probability that a partition is comparable in dominance order to the staircase partition tends to zero as the staircase partition grows.  Secondly, for partitions whose Durfee size is $k$ where $k\geq3$,  by semigroup property, we show that there exists a number $n_k$ such that
if the tensor squares of the first $n_k$ staircase partitions contain all irreducible representations corresponding to partitions with Durfee size $k$, then all tensor squares contain partitions with Durfee size $k$. Specially, we show that $n_3=14$ and $n_4=28$. Furthermore, with the help of computer we show that
the Saxl Conjecture is true for all triple hooks (i.e. partitions with Durfee size 3). Similar results for chopped square and caret shapes are also discussed.
\end{abstract}

\maketitle

\section{Introduction}\label{se:intro}
In representation theory and related fields, the Kronecker  coefficients play a crucial role.
For partitions $\lambda$, $\mu\vdash n$, let $[\lambda]$ and $[\mu]$ be two irreducible representations of $S_n$. The tensor product $[\lambda]\otimes[\mu]$ is an $S_n$-representation via the diagonal embedding $\pi\mapsto (\pi,\pi)$, $\pi\in S_n$. This $S_n$-representation decomposes as follows
$$[\lambda]\otimes[\mu]=\bigoplus_{\nu\vdash n}g(\lambda,\mu,\nu)[\nu],$$
where  $g(\lambda,\mu,\nu)$ are called \emph{Kronecker coefficients}.
In spite of their importance, little is known about the Kronecker coefficients, leaving some fundamental questions unanswered. For example,  no combinatorial description akin to the Littlewood-Richardson rule is known for the Kronecker coefficients.  Another important question is to determine whether they are positive or not, such as the Saxl Conjecture.

In 2012, J. Saxl conjectured that all irreducible representations of the symmetric group occur in the decomposition of the tensor square of the irreducible representation
corresponding to the staircase partition \cite{PPV}. Let $\rho_m$ denote the staircase partition. So the Saxl Conjecture claims that  $g(\rho_m,\rho_m,\lambda)>0$ for each $\lambda\vdash m(m+1)/2$. Many progresses have been made on this conjecture, see for example \cite{Bessenrodt17,Bessenrodt19,Ikenm,Luo,PPV,PP-1,Sellke}.

For $\lambda\vdash m(m+1)/2$, we say that $\lambda$ satisfies Saxl Conjecture if $g(\rho_m,\rho_m,\lambda)>0$.  In \cite[Thm. 2.1]{Ikenm}, Ikenmeyer showed that if a partition $\nu\vdash m(m+1)/2$  is comparable in the dominance order to the staircase partition $\rho_m$, then $\nu$ satisfies Saxl Conjecture.
From his result, we would like to know the proportion of these partitions in the total partitions of $m(m+1)/2$. In Section \ref{se:domi}, by the result of \cite{BP18}, we show that the proportion tends to zero as $m\to \infty$ (see Corollary \ref{cor:clpn}).
Thus the probability that a partition is comparable to the staircase partition tends to zero as  $m\to \infty$.

Another criterion to find partitions satisfying Saxl Conjecture is based on nonvanishing irreducible characters, see for example \cite[Lemma 1.3]{PPV} and \cite[Cor. 4.4]{Bessenrodt17}.
Based on the character criterion, Bessenrodt showed that all double-hooks (i.e. partitions with Durfee size 2) satisfy Saxl Conjecture \cite[Thm. 4.10]{Bessenrodt17}. Recently, by the results of 2-modular representation theory, Bessenrodt et al. verified
Saxl's conjecture for several large new families of partitions, such as partitions which label height 0 characters and $k$-Carter-Saxl pairs \cite{Bessenrodt19}.

In Section \ref{se:fixed}, we give a generalization of Bessenrodt's result in \cite{Bessenrodt17}. For partitions whose Durfee size is $k$ where $k\geq3$,  by semigroup property, we show that there exists a number $n_k$ such that if the tensor squares of the first $n_k$ staircase partitions contain all irreducible representations corresponding to partitions with Durfee size $k$, then the tensor squares of any staircase partitions contain all partitions (of the same weight) with Durfee size $k$. For example, we show that $n_3=14$ and $n_4=28$. With the help of computer,
for triple hooks (i.e. partitions with Durfee size 3), we verify that  $n_3=14$  can be reduced to 9. Combining with the result of \cite{Luo}, we show that all triple hooks satisfy Saxl Conjecture (see Theorem \ref{thm:dm3}).

Our technique is  elementary and based on the semigroup property of Kronecker coefficients \cite{Christ,Luo}.
We also use the technique to discuss the occurrences of hooks and double-hooks for other self-conjugate partitions, such as the chopped square and caret shapes. Our main idea is as follows.
Let $\lambda\vdash n$ be a self-conjugate partition, such as $\rho_m$. For $\mu\vdash n$, we want to determine if $g(\lambda,\lambda,\mu)>0$. Then we reduce the problem of
deciding $g(\lambda,\lambda,\mu)>0$ to deciding $g(\alpha,\alpha,\nu)>0$, where $\alpha\subseteq\lambda$ is a smaller self-conjugate partition.

The paper is organized as follows. In Section \ref{se:pre},  we summarize basic definitions and results needed in this paper. In Section \ref{se:domi}, we show that the probability that a partition is comparable to $\rho_m$ tends to zero as $m\to\infty$.
In Section \ref{se:fixed}, we discuss the occurrences of partitions with fixed Durfee sizes in tensor squares.
With the help of computer, we show that triple hooks satisfy Saxl Conjecture. Some remarks, problems and a generalised Saxl Conjecture for each $n$  are raised in Section \ref{se:rem}.

In this paper, `with the help of computer' means the results are obtained by using Stembridge's Maple package SF \cite{stem}.

\section{Preliminaries}\label{se:pre}

If $A$ is a set, the cardinality of $A$ is denoted by $|A|$.
A partition $\lambda$ of $n$, denoted by $\lambda\vdash n$, is defined to be a weakly decreasing sequence $\lambda= (\lambda_1,\lambda_2,\ldots)$ of non-negative integers such that the sum $\sum_{i} \lambda_i=n$. We also write $\lambda= (\lambda_1,\lambda_2,\ldots,\lambda_k)$ if $\lambda_{k+1}=\lambda_{k+2}=\cdots=0$. Thus for example
$$(3,3, 2, 1, 0, 0, 0,\ldots)=(3, 3, 2, 1, 0, 0)=(3, 3, 2, 1),$$
as partitions of 9.
The \emph{length} of a partition $\lambda$ is the number of its nonzero entries and is denoted by
$\ell(\lambda)$. The set of all partitions of $n$ is denoted by $P(n)$.
To a partition $\lambda$ we associate its \emph{Young diagram}, which is a top-aligned and left-aligned array of boxes such that in row $i$ we have $\lambda_i$ boxes. Thus for  $\lambda\vdash n$ the corresponding Young diagram has $n$ boxes. For example, for $\lambda=(8,8,8,7,7,4)$ the corresponding Young diagram is

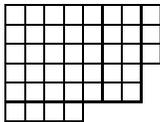
\ytableausetup{smalltableaux}
\centerline{\ydiagram{8,8,8,7,7,4}}
We do not distinguish between a partition $\lambda$ and its Young diagram.
If we transpose a Young diagram at the main diagonal, then we obtain another Young diagram, which is called the \emph{conjugate partition} of $\lambda$ and denoted by $\lambda'$. The row lengths of $\lambda'$ are the column lengths of $\lambda$. In the example above we have $\lambda'=(6,6,6,6,5,5,5,3)$. A partition $\lambda$ is called \emph{self-conjugate} if $\lambda=\lambda'$. Sometimes we use the notation which indicates the number of times each integer occurs as a part of a partition. For example, we write $\lambda=(3,3,3,2,2,1)$  as $(3^3,2^2,1)$  which means that 3  parts of $\lambda$ are equal to $3$, and so on.
We denote by $d(\lambda)$ the Durfee size of $\lambda$, i.e. the number of boxes in the main diagonal of $\lambda$. Let $D(n,k)=\{\mu\in P(n)\mid d(\mu)=k\}$ denote partitions in $P(n)$ whose Durfee size is $k$. Specially, $D\left(\frac{m(m+1)}{2},k\right)$ will be abbreviated as $S(m,k)$.
If the boxes are arranged using matrix coordinates, the \emph{hook} of box $(i,j)$ in a Young diagram is given by the box itself, the boxes to its right and below and is denoted by $h_{i,j}$. The \emph{hook length} is the number of boxes in a hook and denoted by $|h_{i,j}|$.
Define the \emph{principal hook partition} by $\widehat{\lambda}=(|h_{1,1}|,...,|h_{s,s}|)$, where $s = d(\lambda)$. So for $\lambda=(8,8,8,7,7,4)$ above, we have $d(\lambda)=5$ and $\widehat{\lambda}=(13,11,9,6,3)$.
For $m \geq 1$, we call $\rho_m=(m,m-1,...,1)$  the \emph{staircase partition} which is a partition of $\frac{m(m+1)}{2}$.

For $n\in \mathbb{N}$, let $S_n$ denote the symmetric group on $n$ symbols. For a partition $\lambda\vdash n$, let $[\lambda]$ denote the irreducible $S_n$-representation of type $\lambda$. The corresponding irreducible character is denoted by $\chi^\lambda$. For $\nu\vdash n$, let $\chi^\lambda(\nu)$ denote the value of  $\chi^\lambda$ on the conjugacy class of cycle type $\nu$ of the symmetric group $S_n$.

A partition $\lambda$ \emph{dominates} another partition $\mu$, denoted by $\lambda\unrhd\mu$ if for all $k$ we have $\sum_{i=1}^k \lambda_i \geq  \sum_{i=1}^k \mu_i$. If $\lambda$ dominates $\mu$  or $\mu$ dominates $\lambda$, we say that $\lambda$ and $\mu$ are \emph{comparable} in the dominance order. For $\lambda\in P(n)$, let $C(\lambda)\subseteq P(n)$ be the set of partitions which are comparable to $\lambda$. Let $\Lambda(\lambda)\subseteq C(\lambda)$ (resp. $V(\lambda)$) be the set of partitions which are less than or equal to (resp. greater than or equal to) $\lambda$ in dominance order. If $\lambda\trianglelefteq\mu$, then we have $\lambda'\trianglerighteq\mu'$ \cite[Lem. 1.4.11]{JKer81}.
For two partitions $\lambda$ and $\mu$,  let $\lambda+\mu=(\lambda_1+\mu_1, \lambda_2+\mu_2,...)$ and $\lambda-\mu=(\lambda_1-\mu_1, \lambda_2-\mu_2,...)$ denote their rowwise sum and difference. We say $\lambda\subseteq\mu$ if $\lambda_i\leq\mu_i$ for all $i$.

\section{Dominance Order and Saxl Conjecture}\label{se:domi}

Using the result of \cite{BP18}, in this section we will show that the probability that a partition is comparable to $\rho_m$ tends to zero as $m\to \infty$.  It reflects the effectiveness of Ikenmeyer's criterion \cite[Thm. 2.1]{Ikenm}.

Denote by  $\Phi(\rho_m)$ the set of $\lambda\vdash \frac{m(m+1)}{2}$ such that $g(\rho_m,\rho_m,\lambda)>0$. For $\lambda=(\lambda_1, \lambda_2,...)$, if $a\geq \lambda_1$, then  the partition $(a,\lambda_1, \lambda_2,...)$ is denoted by $(a, \lambda)$. Similarly, $(\lambda,1^a)$ denotes the partition $(\lambda_1,...,\lambda_{\ell(\lambda)},1,...,1)$ where  there are $a$ ones behind $\lambda_{\ell(\lambda)}$.
Comparing with Proposition 4.14 of \cite{PPV},  the following gives another lower bound of  $|\Phi(\rho_m)|$.
\begin{proposition}
For $m\geq 3$, there exist at least $2^{m}$ partitions that are comparable to $\rho_m$. In particular, we have $|\Phi(\rho_m)|>2^{m}$.
\end{proposition}
\begin{proof}
We  will show  by induction that there are $2^{m-1}$ partitions less than $\rho_m$. Then by taking transpose, we obtain another $2^{m-1}$ partitions that are greater than $\rho_m$.

There are five partitions less than $\rho_3$: (2,2,1,1), (2,2,2), (2,$1^4$),(3,$1^3$),($1^6$).
Assume that there are at least $2^{m-2}$ partitions less than $\rho_{m-1}$.   For each $\lambda\in \Lambda(\rho_{m-1})$, define two partitions by $(m,\lambda)$ and $(\lambda,1^m)$. Then it is not hard to see that they belong to $\Lambda(\rho_{m})$. Moreover, we can see that for $\lambda,~\mu\in \Lambda(\rho_{m-1})$ if $\lambda\neq\mu$, then
$(m,\lambda)$, $(\lambda,1^m)$, $(m,\mu)$ and $(\mu,1^m)$ are pairwise different.
Thus, for each partition in $\Lambda(\rho_{m-1})$ we obtain two new partitions in $\Lambda(\rho_{m})$ which are pairwise different. So by induction, there are at least $2^{m-1}$ partitions less than $\rho_{m}$.

The lower bound $|\Phi(\rho_m)|>2^{m}$ follows from Theorem 2.1 of \cite{Ikenm}.
\end{proof}

For $\lambda\in P(n)$, if $\sum_{j=1}^{i}\lambda_{j}'\geq \sum_{j=1}^{i}\lambda_{j}+i$, then $\lambda$ is said to be \emph{graphical} \cite{BP18}.
If $\lambda\trianglelefteq\lambda'$, then $\lambda$ is said to be \emph{conjugate-upward}.
Let $G(n)$ and $U(n)$ denote the set of all graphical and conjugate-upward partitions, respectively.
The following theorem gives an upper bound for $|G(n)|/|P(n)|$, which is also suitable for $|U(n)|/|P(n)|$ (see the discussions in \cite[Sect. 1]{BP18}).

\begin{theorem}\cite[Thm. 3.1]{BP18}\label{thm:ugpn}
For $G(n)$, $U(n)$ and $n$ large enough, we have
$$\frac{|U(n)|}{|P(n)|},~\frac{|G(n)|}{|P(n)|}\leq \exp\left(-\frac{0.11\log n}{\log\log n}\right).$$
\end{theorem}

For $C(\lambda)$ defined in Section \ref{se:pre}, by Theorem \ref{thm:ugpn} we have the following corollary.
\begin{corollary}\label{cor:clpn}
Suppose that $\lambda\in P(n)$ is self-conjugate. Then
$$\lim_{n\to +\infty}\frac{|C(\lambda)|}{|P(n)|}=0.$$
In particular, for $\rho_m$ we have
$$\lim_{m\to +\infty}\frac{|C(\rho_m)|}{|P(\frac{m(m+1)}{2})|}=0.$$ That is, the probability that a partition is comparable to $\rho_m$ is zero as $m\to \infty$.
\end{corollary}

\begin{proof}
By definition we have $C(\lambda)=V(\lambda) \cup \Lambda(\lambda)$.  For any $\mu\in \Lambda(\lambda)$, we have that $\mu\trianglelefteq\lambda$ and therefore $\mu'\unrhd\lambda'$. Since $\lambda'=\lambda$, we have $\mu\unlhd\mu'$. Thus, $\Lambda(\lambda)\subseteq U(n)$ and there is a bijection between $V(\lambda)$ and $\Lambda(\lambda)$ by taking transpose .
So  we have that $|C(\lambda)|=2|\Lambda(\lambda)|-1$.

By Theorem \ref{thm:ugpn} we have $\lim_{n\to +\infty}\frac{|U(n)|}{|P(n)|}=0$.
Since $\Lambda(\lambda)\subseteq U(n)$, we  have $\lim_{n\to +\infty}\frac{|\Lambda(\lambda)|}{|P(n)|}=0$ and
$$\lim_{n\to +\infty}\frac{|C(\lambda)|}{|P(n)|}=\lim_{n\to +\infty}\frac{2|\Lambda(\lambda)|-1}{|P(n)|}=0.$$
\end{proof}

\section{Partitions with fixed Durfee sizes in tensor squares}\label{se:fixed}

In this section, for partitions whose Durfee size is $k$ where $k\geq3$,  by semigroup property, we show that there exists a number $n_k$ such that if the tensor squares of the first $n_k$ staircase partitions contain all irreducible representations corresponding to partitions with Durfee size $k$, then the tensor squares of any staircase partitions contain all partitions (of the same weight) with Durfee size $k$ (see Proposition \ref{prp:dmuk} below). Specially, we show $n_3=14$ and $n_4=28$. After that, with the help of computer, we show that all triple hooks satisfy Saxl Conjecture.
We also discuss the occurrences of hooks and double-hooks in the the tensor squares of chopped square and caret shapes.

\subsection{The number $n_k$ via semigroup property}\label{subse:upper}
\

For $\mu\in D(n,k)$, besides the first $k$ columns,  let $\mathcal{A}_i$ denote the set of columns with length $i$. Besides the first $k$ rows, let $\mathcal{B}_i$ denote the set of rows with length $i$. So we have $1\leq i\leq k$.
With these notations, we give the following definition.
\begin{definition}\label{def:alw}
Let $a_{i}$ (resp. $b_{i}$) denote the number of columns (resp. rows) in $\mathcal{A}_i$ (resp. $\mathcal{B}_i$).
The \emph{arm weight} of $\mu\in D(n,k)$ is defined as $A:=\sum_{i=1}^{k}ia_i$. The \emph{leg weight} of $\mu$ is defined as $B:=\sum_{i=1}^{k}ib_i$. If we set $A_i=ia_i$ and $B_i=ib_i$, then $A=\sum_{i=1}^kA_i$ and $B=\sum_{i=1}^kB_i$. Moreover, we have $n=k^2+A+B$.
\end{definition}

For Kronecker coefficients, we have the following property which is called the \emph{semigroup property}: if $g(\lambda,\mu,\nu)>0$ and $g(\alpha,\beta,\gamma)>0$, then $g(\lambda+\alpha,\mu+\beta,\nu+\gamma)>0$ \cite{Christ,Luo}.
Hence, the set of nonzero Kronecker coefficients is a semigroup.
\begin{definition}\label{def:abk}
Let $K$ be the set of all pairs of partitions $(\alpha,\beta)$ such that $g(\alpha,\alpha,\beta)>0$.
\end{definition}

From Definition \ref{def:abk}, we have that if $(\alpha,\beta)\in K$ and $(\lambda,\mu)\in K$, then $(\alpha+\lambda,\beta+\mu)\in K$. The following definition will be used in Subsection \ref{subse:dec10} and Subsection \ref{subsec:app}.

\begin{definition}\cite{Mac,Luo}
Let $\lambda$ and $\mu$ be two partitions. Let
$\lambda\cup\mu$ to be the partition whose parts are those of $\lambda$ and $\mu$, arranged in descending order, which is called the \emph{vertical sum} of $\lambda$ and $\mu$.
\end{definition}
For example, if $\lambda=(3,2,1)$ and $\mu = (2,2)$, then
$\lambda\cup\mu= (3,2,2,2,1)$.
In \cite[Def. 9]{Luo}, the symbol of vertical sum is `$+_V$'. For Young diagrams, we can see that $\lambda\cup\mu$ means they add together vertically.
\begin{lemma}\cite[Cor. 2.5]{Luo}\label{lem:vs}
 If $g(\alpha,\beta,\gamma)>0$ and $g(\lambda,\mu,\nu)>0$, then we have $g(\alpha\cup\lambda,\beta\cup\mu,\gamma+\nu)>0$. In particular, if $(\alpha,\beta)\in K$ and $(\lambda,\mu)\in K$, then $(\alpha\cup\lambda,\beta+\mu)\in K$.
\end{lemma}

In the following, we let
$\tau_{m}^{i}=(m,m-1,...,m-i+1)$ where $m\geq i-1$. Let $\sigma_{m}^{i}=(i^{m-i+1},i-1,i-2,...,2,1)$ denote the conjugate of $\tau_{m}^{i}$.
\begin{definition}
Suppose that $\mu\in S(m,k)$ and $\sigma_{m}^{i},~\upsilon\vdash mi-\frac{i(i-1)}{2}$ for $m\geq i-1$. We say that $\mu$ is  $i$-\emph{decomposable} for $\upsilon$ if $(\sigma_{m}^{i},\upsilon)\in K$ and there exists a partition $\alpha$ such that $\mu=\upsilon+\alpha$, that is, $\mu-\upsilon$ is still a partition.
\end{definition}

It is well known that the Kronecker coefficients are invariant when two of its three partitions are transposed (see e. g. Lemma 2.2 and 2.3 in \cite{Luo}).
\begin{lemma}\label{lem:conj}
For  Kronecker coefficient $g(\lambda,\mu,\nu)$,  we have
 $g(\lambda,\mu,\nu)=g(\lambda',\mu,\nu')=g(\lambda,\mu',\nu')=g(\lambda',\mu',\nu)$.
In particular, we have $g(\tau_{m}^{i},\tau_{m}^{i},\upsilon)=g(\sigma_{m}^{i},\sigma_{m}^{i},\upsilon)$ where $\upsilon\vdash im-\frac{i(i-1)}{2}$.
\end{lemma}

\begin{lemma}\label{lem:either}
For $\mu\in S(m,k)$, if $\mu'$ is $i$-decomposable for  $\upsilon$ and $(\rho_{m-i},\mu'-\upsilon)\in K$, then we have
$(\rho_m,\mu)\in K$.
On the other hand, if
$\mu$ is $i$-decomposable for $\upsilon$ and $(\rho_{m-i},\mu-\upsilon)\in K$,
then we also have $(\rho_m,\mu)\in K$.
\end{lemma}
\begin{proof}
Suppose that $\mu'$ is $i$-decomposable for  $\upsilon$ and $(\rho_{m-i},\mu'-\upsilon)\in K$. Then we have
$(\sigma_{m}^{i},\upsilon)\in K$. Since $\rho_m=\rho_{m-i}+\sigma_{m}^{i}$, by semigroup property we have
$$(\rho_m,\mu')=
\left(\rho_{m-i}+\sigma_{m}^{i},\mu'-\upsilon+\upsilon\right)\in K.$$
By Lemma \ref{lem:conj}, we have $g(\rho_m,\rho_m,\mu)=g(\rho_m,\rho_m',\mu')=g(\rho_m,\rho_m,\mu')$.
Thus, we have $$(\rho_m,\mu)\in K.$$

Similarly, we  have $(\rho_m,\mu)\in K$ if
$\mu$ is $i$-decomposable for $\upsilon$ and $(\rho_{m-i},\mu-\upsilon)\in K$.
\end{proof}

The following lemma generalizes Theorem 2.1 of \cite{Ikenm}.

\begin{lemma}\cite[Thm. 9.1]{Luo}\label{lem:mmn}
For partitions $\mu,\nu \vdash n $, if $\mu$ has distinct row lengths and $\mu\unlhd\nu$, then $(\mu,\nu)\in K$.
\end{lemma}

By Lemma \ref{lem:conj}, Corollary 1.9 of \cite{Tewari} can be reformulated as follows.
\begin{lemma}\cite[Cor. 1.9]{Tewari}\label{lem:mm-1}
For each $\mu\vdash 2m-1$, if $\ell(\mu)\leq4$, then
$(\sigma_m^2,\mu)\in K$.
\end{lemma}

\begin{lemma}\label{lem:m-1}
If $\mu=(m-1,m-1,m-1)$, then
$(\tau_m^3,\mu)$,  $(\sigma_m^3,\mu)\in K$.
\end{lemma}
\begin{proof}
Let $m-2=3s+t$ where  $t\in \{0,1,2\}$.
Then
\begin{align*}
   \tau_m^3=&(m,m-1,m-2)=\left(3s+t+2,3s+t+1,3s+t\right)\\
       =&(3s,3s,3s)+(t+2,t+1,t)
\end{align*}
and
\begin{align*}
   \mu=&(m-1,m-1,m-1)=(3s+t+1,3s+t+1,3s+t+1)\\
      &=(3s,3s,3s)+(t+1,t+1,t+1).
\end{align*}
By Theorem 4.6 of  \cite{PPV}, we have
$((3,3,3),(3,3,3))\in K$. Thus, by semigroup property we have
$$\Big((3s,3s,3s),(3s,3s,3s)\Big)=\Big(s(3,3,3),s(3,3,3)\Big)\in K.$$
For $t=0,1,2$, by computer we can check that $$\Big((t+2,t+1,t),(t+1,t+1,t+1)\Big)\in K.$$
Thus by semigroup property we have
$$(\tau_m^3,\mu)=\Big((3s,3s,3s)+(t+2,t+1,t), (3s,3s,3s)+(t+1,t+1,t+1)\Big)\in K.$$
By Lemma \ref{lem:conj}, we have $g(\sigma_m^3,\sigma_m^3,\mu)=g(\tau_m^3,\tau_m^3,\mu)$ which completes the proof.
\end{proof}

\begin{lemma}\label{lem:mi}
For each $i$, if $\tau_{m}^{i}\unlhd\upsilon$, then $g(\tau_{m}^{i},\tau_{m}^{i},\upsilon)=g(\sigma_{m}^{i},\sigma_{m}^{i},\upsilon)$ and $(\tau_{m}^{i},\upsilon)$, $(\sigma_{m}^{i},\upsilon)\in K$. In particular, we have
\begin{enumerate}
\item\label{itm:mi-3}  if $m$ is odd, then
   $(\sigma_{m}^{3},\upsilon)\in K$ where $\upsilon=(\frac{3m-3}{2},\frac{3m-3}{2})$.
  \item\label{itm:mi-4} $(\sigma_{m}^{4},\upsilon)\in K$ where $\upsilon=(2m-3,2m-3)$;
  \item\label{itm:mi-6o} if $m$ is odd, then $(\sigma_{m}^{6},\upsilon)\in K$ where $\upsilon=(\frac{3m-3}{2}-2,\frac{3m-3}{2}-2,\frac{3m-3}{2}-2,
      \frac{3m-3}{2}-3)$;
  \item\label{itm:mi-8} $(\sigma_{m}^{8},\upsilon)\in K$ where $\upsilon=(2m-7,2m-7,2m-7,2m-7)$;
  \item\label{itm:mi-i} $(\sigma_{m}^{2i},\upsilon)\in K$ where $\upsilon=(2m-2i+1,2m-2i+1,...,2m-2i+1)\vdash 2im-2i^2+i$ and $m\geq 2i$.
\end{enumerate}
\end{lemma}

\begin{proof}
It follows by Lemma \ref{lem:mmn} and \ref{lem:conj}.
\end{proof}

Recall that $D\left(\frac{m(m+1)}{2},k\right)$ is abbreviated as $S(m,k)$.
The upper bound $4k^2+4k-2$ in the following proposition is not best. For $k=3,4$,  we will improve it in Proposition \ref{prop:dmu3} and Proposition \ref{prp:dmu4}.

\begin{proposition}\label{prp:dmuk}
Suppose that  $(\rho_m,\mu)\in K$ for all $m$ such that $1\leq m \leq4k^2+4k-2$ and all $\mu \in S(m,k)$.
Then for all $m\geq 4k^2+4k-1$ and $\mu \in S(m,k)$ we also have $(\rho_m,\mu)\in K$.
\end{proposition}
\begin{proof}
With notations in Definition \ref{def:alw}, we have
\begin{equation}\label{eq:mkab}
\frac{m(m+1)}{2}=k^2+\sum_{i=1}^{k}ia_i+\sum_{i=1}^{k}ib_i.
\end{equation}

Suppose that $a_i\geq 2m-2i+1$ for some $i$. Let $\tau=(2m-2i+1,2m-2i+1,...,2m-2i+1)\vdash 2im-2i^2+i$. Then we can see that $\mu-\tau$ is still a partition. By (\ref{itm:mi-i}) of Lemma \ref{lem:mi} we have $(\sigma_{m}^{2i},\tau)\in K$. Hence, if $(\rho_{m-2i},\mu-\tau)\in K$, then by semigroup property we have $$(\rho_m,\mu)=\left(\rho_{m-2i}+\sigma_{m}^{2i},\mu-\tau+\tau\right)\in K.$$
Similarly, suppose that  $b_i\geq 2m-2i+1$ for some $i$. Then $\mu'-\tau$ is still a partition. If
$(\rho_{m-2i},\mu'-\tau)>0$, then  we have
$$(\rho_m,\mu')=\left(\rho_{m-2i}+\sigma_{m}^{2i},\mu'-\tau+\tau\right)
\in K.$$
Since $g(\rho_m,\rho_m,\mu)=g(\rho_m,\rho_m,\mu')$, we have $(\rho_m,\mu)\in K$.
Hence, by semigroup property if $a_i$ or $b_i\geq 2m-2i+1$ for some $i$, then the positivity of $g(\rho_m,\rho_m,\mu)$ can be reduced to the positivity of $g(\rho_{m-2i},\rho_{m-2i},\mu-\tau)$ or $g(\rho_{m-2i},\rho_{m-2i},\mu'-\tau)$.

Suppose that  $m\geq 4k^2+4k-1$. Then there exists some $i$ such that $a_i$ or $b_i$ $\geq 2m-2i+1$. Otherwise, for each $i$  both $a_i$ and $b_i$ are less than $2m-2i+1$. By (\ref{eq:mkab}) we have
\begin{align*}
  \frac{m(m+1)}{2} &< k^2+\sum_{i=1}^{k}i(2m-2i+1)+\sum_{i=1}^{k}i(2m-2i+1) \\
   &=k^2+2k(k+1)m-\frac{k(k+1)(4k-1)}{3}\\
   &=2k(k+1)m-\frac{4k^3-k}{3},
\end{align*}
which is equivalent to
$$m^2+m(1-4k-4k^2)+\frac{8k^3-2k}{3}<0.$$
It contradicts $m\geq 4k^2+4k-1$.

Hence, it follows from the above considerations that the positivity of $g(\rho_m,\rho_m,\mu)$ for all  $m$ such that $1\leq m \leq 4k^2+4k-2$ and all $\mu \in S(m,k)$ implies the positivity of
$g(\rho_m,\rho_m,\mu)$, where $m\geq 4k^2+4k-1$ and $\mu \in S(m,k)$.
\end{proof}

In the following, we give a proof of Corollary 6.1 in \cite{Ikenm} without using its Theorem 2.1.

\begin{proposition}\label{prp:dnu1}
For every $\nu\vdash m(m+1)/2$, if $d(\nu)=1$ (i.e. $\nu$ is a hook), then we have $(\rho_m,\nu)\in K$.
\end{proposition}
\begin{proof}
Suppose that $\nu\vdash m(m+1)/2$ and $d(\nu)=1$. Then by the notations in Definition \ref{def:alw}, we have
$$\frac{m(m+1)}{2}=|\nu|=1+A+B=1+a_1+b_1.$$
Suppose that $(\rho_i,\mu)\in K$ for all $\mu$ such that $d(\mu)=1$ and all $i$ such that $1\leq i\leq m-1$. We can show that $(\rho_m,\nu)\in K$ by induction.

(1) If $a_1\geq m$, then $\tau=\nu-(m)$ is a partition of $(m-1)m/2$ which is also a hook. Then by induction we have $(\rho_{m-1},\tau)\in K$. Since $((1^m),(m))\in K$, by semigroup property we have
$$(\rho_{m-1}+(1^m),\tau+(m))=(\rho_m,\nu)\in K.$$

(2) If $b_1\geq m$, then $(\rho_m,\nu)\in K$ is equivalent to $(\rho_m,\nu')\in K$ by Lemma \ref{lem:conj}.
By the discussion in (1) above, we also have $(\rho_m,\nu')\in K$.

(3) Suppose that both $a_1$ and $b_1$ are less than $m$. Then we have
\begin{align*}
\frac{m(m+1)}{2}&=1+a_1+b_1\leq1+m-1+m-1\\
 &\leq 2m-1,
\end{align*}
which implies that $m\leq 2$. It is easily checked that  $(\rho_i,\mu)\in K$, where  $d(\mu)=1$ and $i=1,2$.
\end{proof}

Suppose that $\mu\in S(m,3)$. Let $A$ be the arm weight of $\mu$. By Definition \ref{def:alw}, we have $A=a_1+2 a_2+3a_3$ where $a_1$, $a_2$, $a_3\geq0$.
Thus, besides $a_1=a_2=a_3=0$, there are seven cases on the first three rows of $\mu$:
\begin{align}\label{eq:7cases}
&\text{Case (1):}~a_1,a_2,a_3>0;& &\text{Case (2):}~a_1=0,~a_2,a_3>0;\notag\\
&\text{Case (3):}~a_2=0,~a_1,a_3>0;& & \text{Case (4):}~a_3=0,~a_1,a_2>0;\\
&\text{Case (5):}~a_1=a_2=0,~a_3>0;& & \text{Case (6):}~a_1=a_3=0,~a_2>0;\notag\\
&\text{Case (7):}~a_2=a_3=0,~a_1>0.\notag
\end{align}
In the following, we will show the
decomposability of partitions $\mu$ in  $S(m,3)$ and $S(m,4)$ under seven cases in (\ref{eq:7cases}). By Lemma \ref{lem:mm-1}, Lemma \ref{lem:m-1} and Lemma \ref{lem:mi}, the upper bounds in Proposition \ref{prp:dmuk} can be reduced to 14 and 28 for partitions in $S(m,3)$ and $S(m,4)$, respectively. Firstly, we give a definition.

Suppose that $l\leq \frac{m(m+1)}{2}$ and $\mu\in S(m,k)$.
With notations in Definition \ref{def:alw}, if there exists $0\leq x_i\leq a_i$ such that $l=x_1+2x_2+\cdots+kx_k$, then
a partition $\upsilon\vdash l$ can be obtained from the columns of $\mu$ such that $\mu=\tau+\upsilon$ and $\ell(\upsilon)\leq k$, where  $\tau$ is another partition. In fact, we can select $x_i$ columns in $\mathcal{A}_i$ and put them together in their original order. In this way, we obtain the  partition $\upsilon\vdash l$, which is defined as $$\upsilon=\left(\sum_{i=1}^kx_i,\sum_{i=2}^kx_i,\ldots,x_k\right).$$
Specially, when $k=3$ we have $\upsilon=(x_1+x_2+x_3,x_2+x_3,x_3)$.
\begin{definition}
With notations above, we call ($x_1,x_{2},\ldots,x_k$) the \emph{select vector} (or \emph{S-vector} for short) for $\upsilon$.
\end{definition}

\begin{example} Let $m=8$ and $\mu=(14,11,8,3)\in S(8,3)$. We have $a_3=5$, $a_2=3$ and $a_1=3$. Let $x_1=2$, $x_2=2$ and $x_3=3$. We obtain the $S$-vector (2,2,3) and the corresponding partition
$\upsilon=(x_1+x_2+x_3, x_2+x_3, x_3)=(7,5,3)\vdash 2m-1=15$. Let $\tau=(7,6,5,3)$. Then we have $\mu=\tau+\upsilon=(7,6,5,3)+(7,5,3)$, which is described by Young diagrams below.
{\tiny\begin{equation*}
\Yvcentermath1
\yng(14,11,8,3)~=~\yng(7,6,5,3)~+~\yng(7,5,3)
\end{equation*}}
\end{example}

\begin{lemma}\label{lem:w2m-1}
With notations in Definition \ref{def:alw}, suppose that the arm weight $A\geq 2m-1$ for $\mu\in S(m,3)$. Then $\mu$ is $2$-decomposable for some $\upsilon$ if $\mu$ satisfies any of the following conditions:
\begin{enumerate}
\item\label{itm:w2m-1-c7} $a_2=a_3=0$, $a_1>0$;

\item\label{itm:w2m-1-c4} $a_3=0$, $a_1, a_2>0$;

\item\label{itm:w2m-1-c1} $a_1$, $a_2$, $a_3>0$;

\item\label{itm:w2m-1-A23} $0\leq A_2+A_3\leq 2m-1$;

\item\label{itm:w2m-1-c3} $a_2=0$, $a_1,a_3>0$ and $2m-1=3s+t$, where $t\in\{0, 1\}$ and $s\in \mathbb{N}$.
\end{enumerate}
\end{lemma}
We can see that in Lemma \ref{lem:w2m-1} conditions (\ref{itm:w2m-1-c7}), (\ref{itm:w2m-1-c4}), (\ref{itm:w2m-1-c1}) and (\ref{itm:w2m-1-c3}) correspond to Cases (7), (4), (1) and (3) in (\ref{eq:7cases}), respectively.
\begin{proof}
Suppose that $\mu\in S(m,3)$. If a partition $\upsilon\vdash 2m-1$ can be obtained from  $\mu$ such that $\ell(\upsilon)\leq3$ and $\mu=\upsilon+\tau$, then  $(\sigma_{m}^{2}, \upsilon)\in K$ by Lemma \ref{lem:mm-1} and so $\mu$ is $2$-decomposable for $\upsilon$. For five conditions above, the way to get $\upsilon$ is given below.

(1) In this condition, we can choose $2m-1$ columns in $\mathcal{A}_1$. Then we get the partition $\upsilon=(2m-1)$ with $S$-vector ($2m-1$,0, 0).

(2) In this condition, if $2a_2\geq 2m-1$, then  we can
choose $m-1$ columns in $\mathcal{A}_2$ and one column in
$\mathcal{A}_1$. Then we obtain the partition $\upsilon=(m,m-1)$ with $S$-vector (1,$m-1$,0).
If $2a_2<2m-1$, we can
choose $a_2$ columns in $\mathcal{A}_2$ and $2m-1-2a_2$ column in
$\mathcal{A}_1$. Then we obtain the partition $\upsilon=(2m-1-a_2,a_2)$ with $S$-vector ($2m-1-2a_2$, $a_2$, 0).

(3) Since $a_1$, $a_2$, $a_3>0$, we have $\mathcal{A}_1$, $\mathcal{A}_2$ and $\mathcal{A}_3$ are not empty.

\begin{itemize}
\item
Suppose that $3a_3\geq 2m-1$. Write $2m-1$ as $2m-1=3s+t$ where $s\geq0$ and $t\in \{0,1,2\}$.
\begin{itemize}
\item If $t=0$, then extract $s$ columns in $\mathcal{A}_3$ to form $\upsilon=(s,s,s)$ with  $S$-vector is (0, 0, $s$).
\item If $t=1$, then extract $s$ columns in $\mathcal{A}_3$ and one column in $\mathcal{A}_1$ to form $\upsilon=(s+1,s,s)$ with $S$-vector (1, 0, $s$).
\item If $t=2$, then extract $s$ columns in $\mathcal{A}_3$ and one column in $\mathcal{A}_2$ to form $\upsilon=(s+1,s+1,s)$ with  $S$-vector (0, 1, $s$).
\end{itemize}
\item
On the other hand, suppose that $3a_3<2m-1$.
\begin{itemize}
\item
Assume that $2a_2\geq2m-1-3a_3$.
\begin{itemize}
\item[-] If $2m-1-3a_3$ is even, take $\upsilon=(\frac{ 2m-1-a_3}{2}, \frac{2m-1-a_3}{2}, a_3)$ with $S$-vector (0, $\frac{2m-1-3a_3}{2}$, $a_3$).

\item[-]  If $2m-1-3a_3$ is odd, take $\upsilon=(\frac{2m-a_3}{2},\frac{2m-a_3-2}{2}, a_3)$ with $S$-vector (1, $\frac{2m-3a_3-2}{2}$, $a_3$).
\end{itemize}

\item
If $2a_2<2m-1-3a_3$, take
$\upsilon=(u+a_2+a_3, a_2+a_3, a_3)$ with $S$-vector ($u$, $a_2$, $a_3$), where $u=2m-1-3a_3-2a_2$.
\end{itemize}

\end{itemize}

(4) In this condition, we can let $\upsilon=(a_3+a_2+w, a_3+a_2, a_3)$ with $S$-vector $(w,a_2,a_3)$ where $w=2m-1-A_3-A_2$.
In fact, this condition has been shown in (3). Since it will be used often, we put it here separately.

(5) Assume that $t=1$. Then we have $2m-1=3s+1$.
\begin{itemize}
  \item If $A_3\geq 2m-1=3s+1$, take $\upsilon= (s+1,s,s)$ with $S$-vector (1,0,$s$).

  \item If $A_3<2m-1=3s+1$, take  $\upsilon=(a_3+w, a_3, a_3)$ with $S$-vector ($w$, $0$, $a_3$), where $w=2m-1-A_3$.
\end{itemize}
If $t=0$, then we have $2m-1=3s$. Similarly,
if $A_3\geq 2m-1=3s$,  we can let $\upsilon= (s,s,s)$ with $S$-vector (0,0,$s$). If $A_3<2m-1=3s$, we can let $\upsilon=(a_3+w, a_3, a_3)$ with $S$-vector ($w$, $0$, $a_3$), where $w=2m-1-A_3$.
\end{proof}

\begin{lemma}\label{lem:w3m}
Suppose that $\mu\in S(m,3)$. With notations in Definition \ref{def:alw}, $\mu$ is $3$-decomposable for some $\upsilon$ if $\mu$ satisfies any of the following conditions:
\begin{enumerate}
\item\label{itm:w3m-c5} $A_3\geq 3m-3$, especially, when $a_1=a_2=0$, $a_3>0$ and $A\geq 3m-3$;

\item\label{itm:w3m-c3}$a_2=0$, $a_1, a_3>0$ and $A\geq 3m-3$;

\item\label{itm:w3m-c2} $a_1=0$, $a_2, a_3>0$ and $A\geq 3m$.
\end{enumerate}
\end{lemma}
We can see that in Lemma \ref{lem:w3m} conditions (\ref{itm:w3m-c5}), (\ref{itm:w3m-c3}) and (\ref{itm:w3m-c2}) correspond to Cases (5), (3) and (2) in (\ref{eq:7cases}), respectively.

\begin{proof}

(1)
Suppose that $A_3\geq 3m-3$. Then we can
choose $m-1$ columns in $\mathcal{A}_3$ and obtain the partition $\upsilon=(m-1,m-1,m-1)$ with $S$-vector (0, 0, $m-1$).
By Lemma \ref{lem:m-1}, we can see that $\mu$ is $3$-decomposable for $\upsilon$.

(2) In this condition, if $3a_3=A_3\geq 3m-3$, then by (1) we
have that $\mu$ is $3$-decomposable for $\upsilon=(m-1,m-1,m-1)$. On the other hand, suppose that $3a_3<3m-3$.
Then we can
choose $a_3$ columns in $\mathcal{A}_3$ and $3m-3-3a_3$ columns in $\mathcal{A}_1$. In this way, we obtain the partition $\upsilon=(3m-3-2a_3,a_3,a_3)$ with $S$-vector ($3m-3-3a_3, 0, a_3$). We can see that $\tau_{m}^{3}=(m,m-1,m-2)\trianglelefteq\upsilon$. By Lemma \ref{lem:mi}, we have that $\mu$ is $3$-decomposable for $\upsilon$.

(3) In this condition, if $3a_3=A_3\geq 3m-3$, then by (\ref{itm:w3m-c5}) we have $\mu$ is $3$-decomposable for $\upsilon=(m-1,m-1,m-1)$.

On the other hand, suppose that $3a_3<3m-3$. Firstly, assume that
$3m-3-3a_3$ is even. Then we can choose $a_3$ columns in $\mathcal{A}_3$ and $\frac{3m-3-3a_3}{2}$ columns in $\mathcal{A}_2$. So we obtain the partition $\upsilon=(\frac{3m-3-a_3}{2}, \frac{3m-3-a_3}{2}, a_3)$ with $S$-vector (0, $\frac{3m-3-3a_3}{2}$, $a_3$).
We can see that $\tau_{m}^{3} \trianglelefteq \upsilon$. By Lemma \ref{lem:mi}, we have $\mu$ is $3$-decomposable for $\upsilon$.

Secondly, assume that $3m-3-3a_3$ is odd. So $3m-3-3(a_3-1)=3m-3a_3$ is even. Then we can choose $a_3-1$ columns in $\mathcal{A}_3$ and $\frac{3m-3a_3}{2}$ columns in $\mathcal{A}_2$. So we obtain the partition $\upsilon=(\frac{3m-a_3}{2}-1, \frac{3m-a_3}{2}-1, a_3-1)$ with $S$-vector (0, $\frac{3m-3a_3}{2}$, $a_3-1$). We can see that $\tau_{m}^{3}\trianglelefteq \upsilon$. By Lemma \ref{lem:mi}, we have  $\mu$ is $3$-decomposable for $\upsilon$.
\end{proof}

\begin{lemma}\label{lem:dm3}
For $\mu\in S(m,3)$ and $m\geq6$, if its arm (resp. leg) weight is no less than $4m-6$, then $\mu$ (resp. $\mu'$) is $i$-decomposable for some $\upsilon$ and $i\in \{1,2,3,4\}$.
\end{lemma}
\begin{proof}
By notations in Definition \ref{def:alw}, the arm weight of $\mu$ is $A=\sum_{i=1}^{3}ia_i$. By taking transpose, the leg weight becomes the arm weight, so we only need to show the case when $A\geq 4m-6$.

Since $4m-6\geq 2m-1$, $3m-3$ and $3m$ for $m\geq6$, by Lemma \ref{lem:w2m-1} and \ref{lem:w3m} we only need to discuss Case (6) in (\ref{eq:7cases}). Suppose that $a_1=a_3=0$ and $a_2>0$. Then $A=A_2\geq 4m-6$ and we can choose $2m-3$ columns in $\mathcal{A}_2$. In this way, we get the partition  $\upsilon=(2m-3,2m-3)$ with $S$-vector $(0,2m-3,0)$. We can see that $\tau_m^4=(m,m-1,m-2,m-3)\trianglelefteq (2m-3,2m-3)$. Thus, by (\ref{itm:mi-4}) of Lemma \ref{lem:mi} we have that $\mu$ is $4$-decomposable for $\upsilon$.
\end{proof}


\begin{proposition}\label{prop:dmu3}
Suppose that  $(\rho_m,\mu)\in K$ for all $m$ such that $1\leq m \leq 14$ and all $\mu \in S(m,3)$.
Then for all $m\geq 15$ and $\mu \in S(m,3)$, we also have $(\rho_m,\mu)\in K$.
\end{proposition}

\begin{proof}
Let $A$ (resp. $B$) be the arm (resp. leg) weight of $\mu$.
By definition we have
$$\frac{m(m+1)}{2}=3^2+A+B.$$
If $m\geq15$, then either $A$ or $B\geq 4m-6$. Otherwise, if both $A$ and $B$ are less than $4m-6$, then we have
$$\frac{m(m+1)}{2}=9+A+B\leq9+4m-7+4m-7,$$
which is equivalent to
$$m^2-15m+10\leq0.$$
It contradicts $m\geq15$.
Hence, by Lemma \ref{lem:dm3} if $m\geq15$, then either $\mu$ or $\mu'$ is $i$-decomposable for some $\upsilon$, where $i\in\{1,2,3,4\}$.

Suppose that  $(\rho_m, \mu)\in K$ for all $1\leq m \leq 14$ and $\mu \in S(m,3)$. By Lemma \ref{lem:either} and induction, we can see that for all $m\geq 15$ and $\mu \in  S(m,3)$ we also have $(\rho_m,\mu)\in K$.
\end{proof}

\begin{lemma}\label{lem:dm4}
For $\mu\in S(m,4)$ and $m\geq 11$,
if its arm (resp. leg) weight is no less than $8m-28$, then there exist some $i$ and $\upsilon$, where $i\in \{1,2,3,...,8\}$, such that $\mu$ (resp. $\mu'$) is $i$-decomposable for $\upsilon$.
\end{lemma}

\begin{proof}
By definition, the arm weight of $\mu$ is $A=\sum_{i=1}^{4}ia_i$. As in Lemma \ref{lem:dm3}, we only need to show the case when $A\geq 8m-28$.

Firstly, suppose that $A_4=4a_4\geq 2m-1$. There are 8 cases on $a_1$, $a_2$, $a_3$, which consist of 7 cases in (\ref{eq:7cases}) and $a_1=a_2=a_3=0$. We will discuss the decomposability of $\mu\in S(m,4)$ under them.
One the one hand, for $a_1$, $a_2$ and $a_3$, if at least two of them are nonzero, then we can find a $\upsilon\vdash 2m-1$ such that $\mu$ is $2$-decomposable for $\upsilon$. The discussion is given below.

\begin{enumerate}\renewcommand{\labelenumi}{(\theenumi)}
\item Suppose that $a_1$, $a_2$, $a_3>0$. For $A_4=4a_4\geq 2m-1$,
\begin{itemize}
\item if $2m-1=4s+1$, then extract $s$ columns in $\mathcal{A}_4$ and 1 column in $\mathcal{A}_1$ to form $\upsilon=(s+1,s,s,s)$ with $S$-vector $(1, 0, 0, s)$;

\item if $2m-1=4s+3$, then extract $s$ columns in $\mathcal{A}_4$ and 1 column in $\mathcal{A}_3$ to form $\upsilon=(s+1,s+1,s+1,s)$ with $S$-vector $(0, 0, 1, s)$.
\end{itemize}

\item Suppose that $a_1=0$ and $a_2$, $a_3>0$. For $A_4=4a_4\geq 2m-1$,
\begin{itemize}
  \item if $2m-1 = 4s+1 = 4(s-1)+5$, take $\upsilon=(s+1, s+1, s, s-1)$ with $S$-vector $(0, 1, 1, s-1)$;

  \item if $2m-1 = 4s+3$, take $\upsilon=(s+1, s+1, s+1, s)$ with $S$-vector $(0, 0, 1, s)$.
\end{itemize}

\item Suppose that $a_2=0$ and $a_1$, $a_3>0$.
\begin{itemize}
  \item If $2m-1=4s+1$, take $\upsilon=(s+1, s, s, s)$ with $S$-vector $(1, 0, 0, s)$.
  \item If $2m-1=4s+3$, take $\upsilon=(s+1, s+1, s+1, s)$ with $S$-vector $(0, 0, 1, s)$.
\end{itemize}

\item Suppose that $a_3=0$ and $a_1$, $a_2>0$.
\begin{itemize}
  \item If $2m-1=4s+1$, take $\upsilon=(s+1, s, s, s)$ with $S$-vector $(1, 0, 0, s)$.
  \item If $2m-1=4s+3$, take $\upsilon=(s+2, s+1, s, s)$ with $S$-vector $(1, 1, 0, s)$.
\end{itemize}
\end{enumerate}

In four cases above, since $\ell(\upsilon)=4$, by Lemma \ref{lem:mm-1} we have $(\sigma_{m}^{2}, \upsilon)\in K$.
Thus, we have $\mu$ is $2$-decomposable for $\upsilon$.
One the other hand, we  deal with the remaining four cases.

\begin{enumerate}[resume]\renewcommand{\labelenumi}{(\theenumi)} 
\item Suppose that $a_1=a_2=0$ and $a_3>0$.
\begin{itemize}
\item  Firstly, let $a_3\in\{1, 2\}$.
If  $2m-1=4s+3$, take $\upsilon=(s+1,s+1,s+1,s)$ with $S$-vector (0, 0, 1, $s$). By Lemma \ref{lem:mm-1} we have that $(\sigma_{m}^{2}, \upsilon)\in K$. Thus, $\mu$ is $2$-decomposable for $\upsilon$.
If $2m-1=4s+1$, then $m=2s+1$ is odd and $6m-15=4(3s-3)+3$. Now we show that $a_4>3s-3$.
Since $a_1=a_2=0$ and $a_3\in\{1,2\}$, we have $A=A_4+A_3$ and $A_4\geq A-6$. Since $A\geq8m-28$, we have
$$4a_4=A_4\geq A-6\geq8m-34=16s-26.$$
By the assumption $m\geq11$ and $s=\frac{m-1}{2}$, we have  $$a_4\geq \frac{16s-26}{4}>3s-3.$$
Thus, we can take $\upsilon=(3s-2,3s-2,3s-2,3s-3)\vdash 6m-15$ with $S$-vector (0, 0, 1, $3s-3$). By direct computation, $\tau_{m}^6 \unlhd (3s-2,3s-2,3s-2,3s-3)$. So by (\ref{itm:mi-6o}) of Lemma \ref{lem:mi} we have $(\sigma_{m}^{6},\upsilon)\in K$. Thus, $\mu$ is $6$-decomposable for $\upsilon$.

\item  Secondly, suppose that $a_3\geq3$.
 If $2m-1=4s+1$, then $2m-1=4(s-2)+9$. So we can take
$\upsilon=(s+1,s+1,s+1,s-2)\vdash 2m-1$ with $S$-vector (0, 0, 3, $s-2$). By  Lemma \ref{lem:mm-1} we have $(\sigma_{m}^{2},\upsilon)\in K$. Thus, we have that $\mu$ is $2$-decomposable for $\upsilon$.
Similarly, if $2m-1=4s+3$, then $\mu$ is also $2$-decomposable for $\upsilon=(s+1,s+1,s+1,s)$ with $S$-vector $(0,0,1,s)$.
\end{itemize}

\item Suppose that $a_1=a_3=0$ and $a_2>0$. In this case, we have
$A=A_2+A_4$.
\begin{itemize}
\item If $A_4\geq 4m-28$, take $\upsilon=(2m-7,2m-7,2m-7,2m-7)$ with $S$-vector $(0,0,0,2m-7)$. By (\ref{itm:mi-8}) of Lemma \ref{lem:mi} we have  $(\sigma_{m}^{8}, \upsilon)\in K$. Thus, $\mu$ is $8$-decomposable for $\upsilon$.

\item If $A_4=4a_4<4m-28$, take $\upsilon=(2m-a_4-14,2m-a_4-14,a_4,a_4)$ with $S$-vector $(0,2m-2a_4-14,0,a_4)$. Since $\tau_{m}^{8}\unlhd\upsilon$, by Lemma \ref{lem:mi} we have  $(\sigma_{m}^{8}, \upsilon)\in K$.
Thus, $\mu$ is also $8$-decomposable for $\upsilon$.
\end{itemize}

\item Suppose that  $a_2=a_3=0$ and $a_1>0$.
In this case, we have
$A=A_1+A_4$.
\begin{itemize}
\item
If $A_4\geq 4m-28$, then just as the discussion in (6), we have $\mu$ is $8$-decomposable for $\upsilon=(2m-7,2m-7,2m-7,2m-7)$.
\item
If $A_4=4a_4<4m-28$,  take $\upsilon=(4m-3a_4-28,a_4,a_4,a_4)$ with $S$-vector $(4m-4a_4-28,0,0,a_4)$. Since $\tau_{m}^{8}\unlhd\upsilon$, by Lemma \ref{lem:mi} we have  $(\sigma_{m}^{8}, \upsilon)\in K$. Thus, $\mu$ is also $8$-decomposable for $\upsilon$.
\end{itemize}

\item Suppose that  $a_1=a_2=a_3=0$. In this case, we have $A=A_4\geq 8m-28$. So we also have
 $\mu$ is $8$-decomposable for $\upsilon=(2m-7,2m-7,2m-7,2m-7)$.
\end{enumerate}

Secondly, suppose that $4a_4<2m-1$. Then for $m\geq11$ we have
\begin{align*}
3a_3+2a_2+a_1= & A-4a_4\\
>&8m-28-(2m-1) \\
>& 4m-6.
\end{align*}
Hence, besides the columns whose lengths are no less than 4, the remaining columns of $\mu$ satisfy the conditions in Lemma \ref{lem:dm3}. So by the same discussion, we have $\mu$ is $i$-decomposable for some $i\in \{1,2,3,4\}$.
\end{proof}

\begin{proposition}\label{prp:dmu4}
Suppose that  $(\rho_m,\mu)\in K$ for all $m$ such that $1\leq m \leq 28$ and all $\mu \in S(m,4)$.
Then for all $m\geq 29$ and $\mu \in S(m,4)$, we also have $(\rho_m,\mu)\in K$.
\end{proposition}

\begin{proof}
Let $A$ (resp. $B$) be the arm (resp. leg) weight of $\mu$.
By definition we have
$$\frac{m(m+1)}{2}=4^2+A+B.$$
As the proof in Proposition \ref{prop:dmu3}, we have that if $m\geq 29$, then either $A$ or $B\geq 8m-28$.
The proof is completed by Lemma \ref{lem:dm4} and similar discussions as in Proposition \ref{prop:dmu3}.
\end{proof}

\subsection{Triple hooks in tensor squares}\label{se:trip}
\

In this subsection, we will show that all triple hooks satisfy Saxl Conjecture. Firstly, we discuss the decomposability for $\mu\in S(m,3)$, where $10\leq m\leq 14$.

For $\mu\in S(m,3)$, let $A$ (resp. $B$) be the arm (resp. leg) weight of $\mu$.
By Definition \ref{def:alw} we know that
$$A+B=\frac{m(m+1)}{2}-9.$$
By taking transpose, the leg weight becomes the arm weight. So in the following we assume that $A\geq \frac{1}{2}\left(\frac{m(m+1)}{2}-9\right)$.
We will verify the decomposability under the cases in (\ref{eq:7cases}).
We can see that  most cases in (\ref{eq:7cases}) are easy to handle, except Case (2) and Case (6). When $m=10$ and 12, there are some partitions which are hard to decompose. We will treat them individually with the help of computer.

\subsubsection{\textbf{The decomposability of $S(14,3)$}}
\addcontentsline{toc}{subsection}{The decomposability of $S(14,3)$}
\

For $m=14$, since $A+B=m(m+1)/2-9=96$, we assume that $A\geq 48$. Moreover, we have $2m-1=27$, $3m-3=39$ and $4m-6=50$.  So by Lemma \ref{lem:w2m-1} and \ref{lem:w3m}, we find that only Case (6) needs to be discussed.

\textbf{{Case} (6):}
Suppose that $\mu$ satisfies Case (6). Then we have $A=A_2\geq48$. Firstly, if $A_2>48$, then $A_2\geq50=4m-6$. By (\ref{itm:mi-4}) of Lemma \ref{lem:mi} we have that $\mu$ is $4$-decomposable for $\upsilon=(25,25)$.

Secondly, let $A=A_2=48$. Then we have $B=A=48$. Now we discuss the decomposability of $\mu'$. Just as $A=48$, by Lemma \ref{lem:w2m-1} and \ref{lem:w3m}, for $\mu'$ we find that only Case (6) needs to be discussed. That is, $b_1=b_3=0$ and $b_2=24$. So by $A=A_2=48$ we have $a_1=a_3=b_1=b_3=0$ and $a_2=b_2=24$. This implies that $\mu=(27,27,3,2^{24})$. By including the third column of the Durfee square into the arm, we can see that $\mu-\upsilon=(14,14,2^{25})$ is a partition, where $\upsilon=(13,13,1)\vdash 2m-1=27$.
By Lemma \ref{lem:mm-1} we have $(\sigma_{14}^2,\upsilon)\in K$. Thus, $\mu$ is 2-decomposable for $\upsilon$.


\subsubsection{\textbf{The decomposability of $S(13,3)$}}
\addcontentsline{toc}{subsection}{The decomposability of $S(13,3)$}
\

For $m=13$, we have $A+B=m(m+1)/2-9=82$, $2m-1=25$, $3m-3=36$ and $4m-6=46$. Then we assume that $A\geq 41$. So by Lemma \ref{lem:w2m-1} and \ref{lem:w3m}, we find that only Case (6) needs to be discussed.

\textbf{{Case} (6):} Suppose that $\mu$ satisfies Case (6). Then
$A=A_2\geq 41\geq 36=3m-3$. Since $\tau_{13}^{3}\unlhd \upsilon$ where $\upsilon=(18,18)\vdash 36$, by (\ref{itm:mi-3}) of Lemma \ref{lem:mi} we have $\mu$ is $3$-decomposable for $\upsilon$.

\subsubsection{\textbf{The decomposability of $S(12,3)$}}
\addcontentsline{toc}{subsection}{The decomposability of $S(12,3)$}
\

For $m=12$, we have $A+B=m(m+1)/2-9=69$, $2m-1=23$, $3m-3=33$ and $4m-6=42$. Then we assume that $A\geq 35$. So by Lemma \ref{lem:w2m-1} and \ref{lem:w3m}, we find that Case (2) and (6) need to be discussed.

 \textbf{Case (2):}
Under Case (2), we have $A=A_2+A_3$.
\begin{itemize}
\item If  $A_2\geq 2m-1=23$, then
extract 10 columns in $\mathcal{A}_2$ and
1 column  in $\mathcal{A}_3$ to form  $\upsilon=(11,11,1)\vdash23$ with $S$-vector $(0,10,1)$. So by Lemma \ref{lem:mm-1}, we can see that $\mu$ is $2$-decomposable.

\item
If $A_3\geq 2m-1=23$, take $\upsilon=(8,8,7)\vdash23$ with $S$-vector $(0,1,7)$.  By Lemma \ref{lem:mm-1} we have that $\mu$ is also $2$-decomposable.

\item
If both $A_2< 2m-1=23$ and $A_3< 2m-1=23$, by the assumption $A\geq 35$ we have $A_2$, $A_3>12$. Thus, $A_3\in \{15,18,21\}$ and $A_2\in \{14,16,18,20,22\}$. Under this assumption, we can always choose 5 columns  in $\mathcal{A}_3$ and 4
columns  in $\mathcal{A}_2$ to form $\upsilon=(9,9,5)$ with $S$-vector $(0,4,5)$. By Lemma \ref{lem:mm-1} we have that $\mu$ is $2$-decomposable.
\end{itemize}

 \textbf{Case (6):}
Suppose that $\mu$ satisfies Case (6). Then we have $A=A_2\geq 35$. If $A_2\geq 4m-6=42$, then we can choose 21 columns in $\mathcal{A}_2$ to form $\upsilon=(21,21)\vdash42$ with $S$-vector $(0,21,0)$. Since $\tau_{12}^{4}\unlhd \upsilon$,  by (\ref{itm:mi-4}) of Lemma \ref{lem:mi} we have that $\mu$ is $4$-decomposable.

On the other hand, suppose that $42>A_2\geq 35$. Then $A=A_2\in\{36,38,40\}$ and $B\in\{33,31,29\}$.
Firstly, suppose that $b_3=0$. In this case,
we can include the third column of $\mu$ which consists of 3 boxes  into the arm such that $\mu$ is $2$-decomposable for $\upsilon=(11,11,1)$. More precisely, we can choose the third column  and other 10 columns in $\mathcal{A}_2$. Then we get the partition $\upsilon=(11,11,1)$.

Secondly, suppose that $b_3>0$ and recall that $B \in\{33,31,29\}$.
Then we discuss the decomposability of $\mu'$ under the following four conditions on $b_1$, $b_2$, $b_3$. Note that arm weight of $\mu'$ is $B$.

(C-1) $b_1,~b_2>0$, $b_3>0$; (C-2) $b_1=0$, $b_2>0$, $b_3>0$; (C-3) $b_1>0$, $b_2=0$, $b_3>0$; (C-4) $b_1=b_2=0$, $b_3>0$.

\textbf{(C-1):} Suppose that $\mu$ satisfies (C-1).  Since  $B \in\{33,31,29\}$ which is greater than $2m-1=23$,  by (\ref{itm:w2m-1-c1}) of Lemma \ref{lem:w2m-1} we have $\mu'$ is $2$-decomposable.

\begin{table}[h]
\begin{tabular}{|c|c|c|c|c|}
\hline
   $B_2$ & $B_3$ &  $(b_2,b_3)$ &$S$-vector & $\upsilon$\\
\hline
 2 & 27 & (1,9) & (0,1,7)& (8,8,7)\\
\hline
  8 & 21 & (4,7) & (0,4,5) & (9,9,5) \\
\hline
  14 & 15&  (7,5) & (0,4,5) & (9,9,5)\\
\hline
 20 & 9 &  (10,3) & (0,10,1)& (11,11,1)\\
\hline
 26 & 3 &  (13,1) & (0,10,1) & (11,11,1)\\
\hline
\end{tabular}
\caption{The case when $B=B_2+B_3=29$}\label{tab:b2+b3}
\vspace{-1.5em}
\end{table}

\textbf{(C-2):}
Suppose that $\mu$ satisfies (C-2). For each $B\in\{33,31,29\}$, if we write it as $B=3b_3+2b_2$ in all possible ways, we can see that there exist $x_3$ and $x_2$ such that $23=3x_3+2x_2$ and $0\leq x_2\leq b_2$, $0\leq x_3\leq b_3$, that is, we get the  $S$-vector $(0,x_2,x_3)$ for  $\upsilon=(x_3+x_2, x_3+x_2,x_3)\vdash2m-1=23$. Then by Lemma \ref{lem:mm-1} we
have that $\mu'$ is $2$-decomposable for $\upsilon=(x_3+x_2, x_3+x_2,x_3)$. The case when $B=B_2+B_3=29$ is given in Table \ref{tab:b2+b3}. When $B=31$ and $33$, the discussions are similar.

\textbf{(C-3):}
Suppose that $\mu$ satisfies (C-3). Firstly, suppose that  $b_1=1$. Then for $B\in\{33,31,29\}$ we should have $B=31$ and therefore $b_3=10$. Since $A_3=0$, we can include the third row of $\mu$ which consists of 3 boxes into the leg such that $\mu'$ is $3$-decomposable for $\upsilon=(11,11,11)$. In fact, we can choose the third row of $\mu$ and other 10 rows in $\mathcal{B}_3$. After taking transpose, we get the partition $\upsilon=(11,11,11)$ with $S$-vector $(0,0,11)$.

Secondly, suppose that $b_1\geq2$. Then for each $B$ if $B_3\geq 23=2m-1$,  by choosing 7 rows in $\mathcal{B}_3$ and 2 rows in $\mathcal{B}_1$ and taking transpose, we obtain the partition $\upsilon=(9,7,7)\vdash23=2m-1$ with $S$-vector $(2,0,7)$. By Lemma \ref{lem:mm-1} we have $\mu'$ is $2$-decomposable for $\upsilon=(9,7,7)$. On the other hand, suppose that $B_3<23=2m-1$ for  each $B$. Since  $B_2=0$, we have $B_2+B_3=B_3<23=2m-1$. So by (\ref{itm:w2m-1-A23}) of Lemma \ref{lem:w2m-1}, we have $\mu'$ is $2$-decomposable.

\textbf{(C-4):}
Suppose that $\mu$ satisfies (C-4). Then for $B\in\{33,31,29\}$ we should have  $B=B_3=3m-3=33$. By (\ref{itm:w3m-c5}) of Lemma \ref{lem:w3m} we have that $\mu'$ is $3$-decomposable.

\subsubsection{\textbf{The decomposability of $S(11,3)$}}
\addcontentsline{toc}{subsection}{The decomposability of $S(11,3)$}
\

For $m=11$, we have $A+B=m(m+1)/2-9=57$, $2m-1=21$, $3m-3=30$ and $4m-6=38$. Then we assume that $A\geq 29$.
So by Lemma \ref{lem:w2m-1} and \ref{lem:w3m}, we find that
Case (2),   (5) and (6) need to be discussed.

 \textbf{\textbf{Case (2)}:}
Under Case (2), we have $A=A_2+A_3$.
\begin{itemize}
\item
Firstly, if  $A_2\geq 2m-1=21$, extract 9 columns  in $\mathcal{A}_2$ and 1 column  in $\mathcal{A}_3$ to form $\upsilon=(10,10,1)\vdash21$ with $S$-vector $(0,9,1)$. Then by Lemma \ref{lem:mm-1}, we can see that $\mu$ is $2$-decomposable.

\item Secondly, if $A_3\geq 2m-1=21$, take $\upsilon=(7,7,7)\vdash21$ with $S$-vector $(0,0,7)$. By Lemma \ref{lem:mm-1}, we also have that $\mu$ is $2$-decomposable.
\item Thirdly, suppose that both $A_2<2m-1=21$ and $A_3<2m-1=21$. More precisely, it means that $A_2\leq20$ and $A_3\leq18$.  By the assumption $A_2+A_3\geq 29$ we have $A_3\geq9$ and $A_2\geq11$. Thus, $A_3\in\{9,12,15,18\}$ and $A_2\in\{12,14,16,18,20\}$. In this case, we can always extract 3 columns  in $\mathcal{A}_3$ and 6 columns  in $\mathcal{A}_2$ to form $\upsilon=(9,9,3)$ with $S$-vector $(0,6,3)$. By Lemma \ref{lem:mm-1} we have that $\mu$ is $2$-decomposable.
\end{itemize}

\textbf{\textbf{Case (5)}:}
Under Case (5), we have $A=A_3$. Since $A=A_3\geq29>21=2m-1$,
we can choose 7 columns  in $\mathcal{A}_3$ to form $\upsilon=(7,7,7)\vdash21$ with $S$-vector $(0,0,7)$. Then by Lemma \ref{lem:mm-1},  $\mu$ is $2$-decomposable.

\textbf{\textbf{Case (6)}:}
Under Case (6), we have $A=A_2 \geq29$. Since $A_2$ is even, we have $A=A_2\geq30$.
So we can choose 15 columns  in $\mathcal{A}_2$ to form $\upsilon=(15,15)\vdash30$ with $S$-vector $(0,15,0)$.
Since $\tau_{11}^{3}\unlhd(15,15)$, by (\ref{itm:mi-3}) of Lemma \ref{lem:mi} we have $(\sigma_{11}^{3},\upsilon)\in K$. So $\mu$ is $3$-decomposable.

\subsubsection{\textbf{The decomposability of $S(10,3)$}}\label{subse:dec10}
\addcontentsline{toc}{subsection}{The decomposability of $S(10,3)$}
\

For $m=10$, we have $A+B=m(m+1)/2-9=46$, $2m-1=19=3\times6+1$, $3m-3=27$ and $4m-6=34$. Then we assume that $A\geq 23$. So by Lemma \ref{lem:w2m-1} and \ref{lem:w3m}, we find that
Case (2),  (5) and  (6) need to be discussed.
Moreover, Case (2) and (6) are the most involved parts.
Under Case (2) and  (6), there exist partitions which are hard to decompose. They will be dealt with the help of computer.

\begin{table}[h]
\begin{tabular}{|c|c|c|c|}
\hline
   & Cases & $B=23$ and $A_3=21$, $A_2=2$ & $B=20$ and $A_3=24$, $A_2=2$ \\
\hline
$(1)$ & $b_1,b_2,b_3>0$ & By (\ref{itm:w2m-1-c1}) of Lemma \ref{lem:w2m-1} & By (\ref{itm:w2m-1-c1}) of Lemma \ref{lem:w2m-1}  \\
\hline
$(2)$ & $b_1=0$ and $b_2,b_3>0$ & Hard  & Hard   \\
\hline
(3) & $b_2=0$ and $b_1,b_3>0$ & By (\ref{itm:w2m-1-c3}) of Lemma \ref{lem:w2m-1} & By (\ref{itm:w2m-1-c3}) of Lemma \ref{lem:w2m-1} \\
\hline
(4)  & $b_3=0$ and $b_1,b_2>0$ & By (\ref{itm:w2m-1-c4}) of Lemma \ref{lem:w2m-1} & By (\ref{itm:w2m-1-c4}) of Lemma \ref{lem:w2m-1}  \\
\hline
(5) & $b_1=b_2=0$ and $b_3>0$ & $\emptyset$ & $\emptyset$  \\
\hline
(6) & $b_1=b_3=0$ and $b_2>0$ & $\emptyset$ & \tabincell{c}{$\mu$ is 3-decomposable\\for $\upsilon=(9,9,9)$}\\

\hline
(7) & $b_2=b_3=0$ and $b_1>0$ & By (\ref{itm:w2m-1-c7}) of Lemma \ref{lem:w2m-1} & By (\ref{itm:w2m-1-c7}) of Lemma \ref{lem:w2m-1} \\
\hline
\end{tabular}
\caption{Decomposability of $\mu'$ for $\mu\in S(10,3)$, where ``$\emptyset$'' means impossible.}\label{tab:b2b3m10}
\vspace{-1.5em}
\end{table}

 \textbf{\textbf{Case (2)}:}
Under Case (2), we have $A=A_2+A_3\geq 23$.  The discussion will be given in the following conditions:

(\romannumeral1) $A_3\geq19$ and $A_2=2$; (\romannumeral2) $A_3\geq19$ and $A_2\geq4$; (\romannumeral3) $A_2\geq19$; (\romannumeral4) $A_2, A_3<19$.
\\
\noindent\textbf{(\romannumeral1):} Suppose that $\mu$ satisfies condition (\romannumeral1).
Since $A=A_3+A_2\geq23$, we have $A_3\geq 23-A_2=21$. So we have $A_3\in \{21,24,27,...\}$. If $A_3\geq 27=3m-3$, then by (\ref{itm:w3m-c5}) of Lemma \ref{lem:w3m} we have $\mu$ is $3$-decomposable.

Now suppose that $A_3\in\{21,24\}$ and $A_2=2$. Then we have $B\in \{23,20\}$. In both conditions, we will mainly discuss the decomposability of $\mu'$ whose arm weight is $B$.  There are 7 cases on $B$ which are the same as (\ref{eq:7cases}). The discussion is summarized in  Table \ref{tab:b2b3m10}. In the following, we give an explanation of Table \ref{tab:b2b3m10}. We will find that there are 4 partitions which are hard to decompose.

In  Table \ref{tab:b2b3m10}, under Case (1), (3), (4) and (7), the decomposability of $\mu'$ can be obtained from Lemma \ref{lem:w2m-1}. For Case (2), if $B=23$, that is,  $b_1=0$, $b_2$, $b_3>0$, then there are exactly 4 pairs ($b_2$, $b_3$) such that $23=2b_2+3b_3$, which are \{$(10,1)$, $(7,3)$, $(4,5)$, $(1,7)$\}. Suppose that ($b_2$, $b_3$)=$(10,1)$. We have $2m-1=19=2x_2+3x_3$ where $(x_2,x_3)=(8,1)$. Let $(x_1,x_2,x_3)=(0,8,1)$ be the $S$-vector for $\upsilon=(x_3+x_2,x_3+x_2,x_3)=(9,9,1)$. By Lemma \ref{lem:mm-1} we have $\mu'$ is 2-decomposable for $\upsilon=(9,9,1)$.
Similarly,
if ($b_2$, $b_3$)=$(7,3)$, we  can set $(x_1,x_2,x_3)=(0,5,3)$.
If ($b_2$, $b_3$)=$(4,5)$, we can set $(x_1,x_2,x_3)=(0,2,5)$.
However, when ($b_2$, $b_3$)=$(1,7)$ we cannot find $0\leq x_2\leq b_2$, $0\leq x_3\leq b_3$ such that $2m-1=19=2x_2+3x_3$. When ($b_2$, $b_3$)=$(1,7)$, we can see that the corresponding partition is $(11,11,10,3^7,2)$.

On the other hand, suppose that $B=20$ for Case (2). Then there are only 3 pairs ($b_2$, $b_3$) such that $20=2b_2+3b_3$, which are \{$(1,6)$, $(4,4)$, $(7,2)$\}.
We can see that for these 3 pairs, we cannot find $0\leq x_2\leq b_2$, $0\leq x_3\leq b_3$ such that $2m-1=19=2x_2+3x_3$. Moreover, when ($b_2$, $b_3$)=$(1,6)$, $(4,4)$ and $(7,2)$, the corresponding partitions are $(12,12,11,3^6,2)$, $(12,12,11,3^4,2^4)$ and $(12,12,11,3^2,2^7)$, respectively.

So by discussions above, in Table \ref{tab:b2b3m10}  when  $B=23$ and 20 satisfy Case (2), we find 4 partitions which are hard to decompose:
$\mu_1=(11,11,10,3^7,2)$, $\mu_2=(12,12,11,3^6,2)$,
$\mu_3=(12,12,11,3^4,2^4)$ and $\mu_4=(12,12,11,3^2,2^7)$.
Let $\tau_7=(7,7,7)$, $\upsilon_1=(4^2,3^8,2)$, $\upsilon_2=(5^2,4,3^6,2)$, $\upsilon_3=(5^2,4,3^4,2^4)$ and $\upsilon_4=(5^2,4,3^2,2^7)$. Then we have $\mu_i=\tau_7+\upsilon_i$ ($1\leq i\leq4$).
By computer, we can verify that $(\sigma_{10}^{4},\upsilon_i)\in K$ for $1\leq i\leq4$. Thus $\mu_i$ is $4$-decomposable for $\upsilon_i$.

Suppose that $B\in \{23,20\}$. Since 23 and 20 cannot be divided by 3, Case (5) cannot happen which is denoted by ``$\emptyset$'' in Table \ref{tab:b2b3m10}. Since 23 cannot be divided by 2,  Case (6) cannot happen for  $B=23$.
For $B=20$ and $A_3=24$, $A_2=2$, if $B$ satisfies Case (6), that is, $b_1=b_3=0$ and $b_2>0$, then only one partition satisfies this condition, which is $\mu=(12^2,11,2^{10})$. Under this assumption, we discuss the decomposability of $\mu$. By direct computation, we have that $\mu$ is $3$-decomposable for $\upsilon=(9,9,9)$.

\textbf{(\romannumeral2):} Suppose that $\mu$ satisfies condition (\romannumeral2).
Since $A_3\geq19$ and $A_2\geq4$, we can extract 5 columns  in $\mathcal{A}_3$ and $2$ columns in $\mathcal{A}_2$ to form $\upsilon=(7,7,5)\vdash19=2m-1$ with $S$-vector $(0,2,5)$. By Lemma \ref{lem:mm-1} we have $\mu$ is $2$-decomposable.

\textbf{(\romannumeral3):} Suppose that $\mu$ satisfies condition (\romannumeral3).
Since $A_2\geq19$ and $A_3>0$, we can extract 1 column in $\mathcal{A}_3$ and $8$ columns in $\mathcal{A}_2$ to form $\upsilon=(9,9,1)\vdash19=2m-1$ with $S$-vector $(0,8,1)$. By Lemma \ref{lem:mm-1} we have $\mu$ is $2$-decomposable.

\textbf{(\romannumeral4):} Suppose that $\mu$ satisfies condition (\romannumeral4).
For $A_2+A_3\geq 23$, if $A_3$, $A_2<19$, we have $4<A_2, A_3<19$.
More precisely, $A_2\in \{6,8,10,12,14,16,18\}$ and
$A_3\in \{6, 9, 12, 15, 18\}$. The decomposability of $\mu$ is given below.

\textbf{(\romannumeral4-1):} If $A_3=6$, from $A_2+A_3\geq 23$ we should have $A_2=18$. So we can extract 1 column   in $\mathcal{A}_3$ and $8$ columns   in $\mathcal{A}_2$ to form $\upsilon=(9,9,1)\vdash19=2m-1$ with $S$-vector $(0,8,1)$. By Lemma \ref{lem:mm-1} we have $\mu$ is $2$-decomposable for $\upsilon$.

\textbf{(\romannumeral4-2):} If $A_3=9$, from $A_2+A_3\geq 23$ we should have $A_2\in\{14,16,18\}$. So we can extract 3 columns in $\mathcal{A}_3$ and $5$ columns  in $\mathcal{A}_2$ to form $\upsilon=(8,8,3)\vdash19=2m-1$ with $S$-vector $(0,5,3)$. By Lemma \ref{lem:mm-1} we have $\mu$ is $2$-decomposable for $\upsilon$.

\textbf{(\romannumeral4-3):} If $A_3=12$, from $A_2+A_3\geq 23$ we should have $A_2\in\{12,14,16,18\}$. The discussion is the same as \textbf{(\romannumeral4-2)}. That is, $\mu$ is also $2$-decomposable for $\upsilon=(8,8,3)$.

\textbf{(\romannumeral4-4):} If $A_3=15$, from $A_2+A_3\geq 23$ we should have $A_2\in\{8,10,12,14,16,18\}$. So we can extract 5 columns in $\mathcal{A}_3$ and $2$ columns in $\mathcal{A}_2$ to form $\upsilon=(7,7,5)\vdash19=2m-1$ with $S$-vector $(0,2,5)$. By Lemma \ref{lem:mm-1} we have $\mu$ is $2$-decomposable for $\upsilon$.

\textbf{(\romannumeral4-5):} If $A_3=18$, from $A_2+A_3\geq 23$ we should have $A_2\in\{6,8,10,12,14,16,18\}$. The discussion is the same as \textbf{(\romannumeral4-4)}.  That is, $\mu$ is also $2$-decomposable for $\upsilon=(7,7,5)$.

\textbf{\textbf{Case (5)}:}
Suppose that $\mu$ satisfies Case (5). Then $A=A_3\geq23$ which means that $A_3\in\{24, 27, 30\ldots\}$.
If $A_3\geq27=3m-3$, then by (\ref{itm:w3m-c5}) of Lemma \ref{lem:w3m} we have that $\mu$ is $3$-decomposable for $\upsilon=(9,9,9)$.
If $A=A_3=24$, then $B=46-A_3=22$. Under this assumption, we mainly discuss the decomposability of $\mu'$ whose arm weight is $B$.  The discussion is given under 7 cases as in (\ref{eq:7cases}).

\textbf{(b-1)}: Suppose that $b_1,b_2,b_3>0$. Since $B=22\geq19=2m-1$, by (\ref{itm:w2m-1-c1}) of Lemma \ref{lem:w2m-1} we have $\mu'$ is $2$-decomposable.

\textbf{(b-2)}: Suppose that $b_1=0$ and $b_2,b_3>0$.
There are 3 pairs ($b_2$, $b_3$) such that $22=2b_2+3b_3$, which are \{$(2,6)$, $(5,4)$, $(8,2)$\}. If ($b_2$, $b_3$)=$(2,6)$, we have $2m-1=19=2x_2+3x_3$ where $(x_2,x_3)=(2,5)$. So by Lemma \ref{lem:mm-1} we have $\mu'$ is 2-decomposable for $\upsilon=(x_3+x_2,x_3+x_2,x_3)=(7,7,5)$, whose $S$-vector is $(x_1,x_2,x_3)=(0,2,5)$. Similarly,
if ($b_2$, $b_3$)=$(5,4)$, we  can set $(x_1,x_2,x_3)=(0,5,3)$.
If ($b_2$, $b_3$)=$(8,2)$, we can set $(x_1,x_2,x_3)=(0,8,1)$.

\textbf{(b-3)}: Suppose that $b_2=0$ and $b_1,b_3>0$. Since $B=22\geq19=2m-1$ and $19=3\times6+1$, by (\ref{itm:w2m-1-c3}) of Lemma \ref{lem:w2m-1} we have $\mu'$ is $2$-decomposable.

\textbf{(b-4)}: Suppose that $b_3=0$ and $b_1,b_2>0$. Since $B=22\geq19=2m-1$, by (\ref{itm:w2m-1-c4}) of Lemma \ref{lem:w2m-1} we have $\mu'$ is $2$-decomposable.

\textbf{(b-5)}: Suppose that $b_1=b_2=0$ and $b_3>0$. Since $22$ cannot be divided by 3, this case cannot happen.

\textbf{(b-6)}: Suppose that $b_1=b_3=0$ and $b_2>0$. Since $A=A_3=24$, there is only one partition satisfying this case, which is $\mu=(11,11,11,2^{11})$. By direct computation, we have $\mu$ is $3$-decomposable for $\upsilon=(9,9,9)$.

\textbf{(b-7)}: Suppose that $b_2=b_3=0$ and $b_1>0$. Since $B=22\geq19=2m-1$, by (\ref{itm:w2m-1-c7}) of Lemma \ref{lem:w2m-1} we have $\mu'$ is $2$-decomposable.

\textbf{Case (6)}:
Suppose that $\mu$ satisfies Case (6). Then $A=A_2\geq23$. Firstly, if $A_2\geq 34=4m-6$, then by (\ref{itm:mi-4}) of Lemma \ref{lem:mi} we have $\mu$ is $4$-decomposable for $\upsilon=(17,17)$. Secondly, suppose that $23\leq A_2<34$. Then we have $A_2\in\{24,26,28,30,32\}$. Let
$$C_6=\{\mu\in S(10,3)| \mu~\text{satisfies Case (6) and}~23\leq A_2 < 34\}$$
denote the set of these partitions. Generally, by previous method it is hard to verify whether these partitions are $i$-decomposable or not. So in this case, we don't discuss the decomposability of $\mu$. With the help of computer,
we can show that $(\rho_{10},\mu)\in K$ for $\mu\in C_6$. In fact, let $\eta_6=(6,6,6,6,6,5)\vdash35$. Then we have $\rho_{10}=(\eta_6\cup\rho_4)+\rho_4$.
For $\mu\in C_6$, we can see that $\nu=(\mu-\tau_5)-\tau_5$ is still a partition, where $\nu\vdash35$ and $\tau_5=(5,5)$. So we have $\mu=\nu+\tau_5+\tau_5$.
By computer we can check that $(\eta_6,\nu)\in K$ for all $\nu\vdash35$.
For $\rho_4$, we have $(\rho_4,\tau_5)\in K$. Then by Lemma \ref{lem:vs} we have
$$(\eta_6\cup\rho_4,\nu+\tau_5)\in K.$$
And then we have
$$\left((\eta_6\cup\rho_4)+\rho_4,\nu+\tau_5+\tau_5\right)=(\rho_{10},\mu)\in K.$$

By now, we can see that when $11\leq m\leq 14$ and $\mu\in S(m,3)$, there exists some $i\in\{2,3,4\}$ such that $\mu$ or $\mu'$ is $i$-decomposable. However, there are some exceptions when $\mu\in S(10,3)$. Let $C_6$ denote the set of these partitions.
With the help of computer, we checked that $(\rho_{10},\mu)\in K$ for $\mu\in C_6$. Thus, the upper bound 14 in Proposition \ref{prop:dmu3} can be reduced to 9.
In \cite[Sec. 7]{Luo}, using a computer to implement the semigroup property in conjunction with Theorem 2.1 of \cite{Ikenm} on dominance ordering, Luo and Sellke verified the Saxl Conjecture up to $\rho_9$. Thus, from discussions above we have the main theorem of this paper.

\begin{theorem}\label{thm:dm3}
For each $m\in \mathbb{N}$, if $\mu \in S(m,3)$, we have $(\rho_m,\mu)\in K$.
\end{theorem}

\subsection{Applications to chopped square and caret shapes}\label{subsec:app}
\

In this part, we apply our technique for staircase shapes to chopped square and caret shapes, which were defined in \cite{PPV}.
Similar results are obtained for the occurrences of hooks and double-hooks.

Let
$$\eta_k=(k^{k-1}, k-1) \vdash n$$
where $n=k^2-1$.   Let
$$\gamma_k=(3k-1, 3k-3,...,k+3, k+1, k, k-1, k-1, k-2, k-2,..,2,2,1,1) \vdash n$$
where $n=3k^2$. In \cite{PPV},  $\eta_k$ and $\gamma_k$ are called \emph{the chopped square shape of order $k$} and \emph{the caret shape of order $k$}, respectively. It has been shown that hooks and two row shapes appear in $[\eta_k]\otimes [\eta_k]$ and $[\gamma_k]\otimes[\gamma_k]$ for sufficient large $k$ \cite{PPV}.

In the following, we discuss the occurrences of hooks in $[\eta_k]\otimes [\eta_k]$ and $[\gamma_k]\otimes[\gamma_k]$.

\begin{proposition}\label{prp:csdnu1}
For every $\nu\vdash k^2-1$, if $d(\nu)=1$ (i.e. $\nu$ is a hook), then we have $(\eta_k,\nu)\in K$.
\end{proposition}
\begin{proof}
Suppose that $\nu\vdash k^2-1$ and $d(\nu)=1$. Then by the notations in Definition \ref{def:alw}, we have
$$k^2-1=|\nu|=1+A+B=1+a_1+b_1.$$
Suppose that $(\eta_i,\mu)\in K$ for all $d(\mu)=1$ and $1\leq i\leq k-1$. We can show that $(\eta_k,\nu)\in K$ by induction.

If  $A\geq2k-1$, then  $\nu-(2k-1)$ is still a partition.
Suppose that $(\eta_{k-1},\nu-(2k-1))\in K$. Then  by
$\left((1^{k-1}),(k-1)\right)\in K$ we have
$$\left(\eta_{k-1}+(1^{k-1}),\nu-(2k-1)+(k-1)\right)\in K.$$
Let $\theta_{k-1}=\eta_{k-1}+(1^{k-1})$ and $\nu_{k-1}=\nu-(2k-1)+(k-1)$. Then by $\left((k),(k)\right)\in K$ and Lemma \ref{lem:vs} we have
$$\left(\theta_{k-1}\cup(k),
\nu_{k-1}+(k)\right)=(\eta_{k},\nu)\in K.$$
Similarly, if $B\geq2k-1$, from  $g(\eta_{k},\eta_{k},\nu')=g(\eta_{k},\eta_{k},\nu)$
we have $(\eta_{k},\nu)\in K$.

Suppose that both $A$ and $B$ are less than $2k-1$. Then we have
\begin{align*}
k^2-1&=1+A+B\\
 &\leq1+ 2k-2+2k-2\\
 &=4k-3,
\end{align*}
which implies that $k\leq 3$. It has been verified that $(\eta_i,\nu)\in K$  for all $\nu\vdash i^2-1$ with $d(\nu)=1$ and $i\in \{1,2,3\}$ \cite[Sect. 3.4]{PPV}.
\end{proof}

\begin{proposition}\label{prp:cardnu1}
For every $\nu\vdash 3k^2$, if $d(\nu)=1$, then we have $(\gamma_k,\nu)\in K$.
\end{proposition}

\begin{proof}
The proof is  similar  to Proposition \ref{prp:csdnu1}.
Suppose that $\nu\vdash 3k^2$ and $d(\nu)=1$. Then by the notations in Definition \ref{def:alw}, we have
$$3k^2=|\nu|=1+A+B=1+a_1+b_1.$$

Note that $\gamma_k=\left(\gamma_{k-1}+(1^{3k-2})\right)\cup (3k-1)$.
So if  $A\geq 6k-3$, then $(\gamma_{k-1},\nu-(6k-3))\in K$ implies $(\gamma_{k},\nu)\in K$. Similar result is hold for  $B\geq 6k-3$.

If both $A$ and $B$ are less than $6k-3$, we have
\begin{align*}
3k^2&=1+A+B\\
 &\leq1+ 6k-4+6k-4\\
 &=12k-7,
\end{align*}
which implies that $k\leq 3$. By computer, we can easily check that $(\gamma_i,\nu)\in K$ for all $\nu\vdash 3i^2$ with $d(\nu)=1$ and $i\in \{1,2,3\}$.
\end{proof}

In the following, we discuss the positivity of $g(\eta_k,\eta_k,\mu)$ and $g(\gamma_k,\gamma_k,\nu)$, where $\mu\in D(k^2-1,2)$ and $\nu\in D(3k^2,2)$.

\begin{lemma}\label{lem:k22k}
For $k\in \mathbb{N}$, if $k$ is even, we have $((k^2),(k^2))\in K$. If $k$ is odd, we have $g((k^2),(k^2),(k^2))=0$. However, if $k$ is odd, we have $((k^4),(2k,2k))\in K$.
\end{lemma}
\begin{proof}
By Lemma 1.6 of \cite{Tewari}, we have that if $k$ is even, then $g((k^2),(k^2),(k^2))=1>0$ and if $k$ is odd, then $g((k^2),(k^2),(k^2))=0$.

If $k$ is odd, write $k$ as $2s+1$ for some $s$,
then we have $k=2(s-1)+3$ and $2k=4(s-1)+6$.
So we have
\begin{align*}
  (k^4)=&(k,k,k,k)=\Big(2(s-1)+3,2(s-1)+3,2(s-1)+3,2(s-1)+3\Big)\\
   =&(s-1)(2,2,2,2)+(3,3,3,3)=(s-1)(2^4)+(3^4),
\end{align*}
and
$$(2k,2k)=\Big(4(s-1)+6,4(s-1)+6\Big)=(s-1)(4,4)+(6,6).$$
Since $(2^4)'=(4,4)$ and $((4,4),(4,4))\in K$,   we have
$$((2^4),(4,4))\in K$$ and therefore $$((s-1)(2^4),(s-1)(4,4))\in K.$$
By computer we can verify that
$$((3^4),(6,6))\in K.$$
So by semigroup property we have
$$((k^4),(2k,2k))=\Big((s-1)(2^4)+(3^4), (s-1)(4,4)+(6,6)\Big)\in K.$$
\end{proof}

\begin{lemma}\label{lem:df2cse}
Suppose that  $k\geq 12$ is even and $\mu \in D(k^2-1,2)$.
Then deciding $(\eta_k,\mu)\in K$ can be reduced to $\eta_{k-1}$ or $\eta_{k-2}$.
\end{lemma}
\begin{proof}
For $\mu \in D(k^2-1,2)$, with notations in Definition \ref{def:alw} we have
$$k^2-1=|\mu|=4+A+B=4+A_1+A_2+B_1+B_2.$$
If $k\geq 12$, then one of the following conditions must occur: \begin{align}\label{eq-csdf2}
A_2\geq 4k-4,~~B_2\geq 4k-4,~~A_1\geq 2k-1~\text{and}~B_1\geq 2k-1.     \end{align}
Otherwise, we have
\begin{align*}
 k^2-1=& 4+A_1+A_2+B_1+B_2\\
 \leq& 4+2k-2+4k-5+2k-2+4k-5\\
 =&12k-10
\end{align*}
which contradicts $k\geq12$.

Suppose that $k$ is even.
If $\mu \in D(k^2-1,2)$ satisfies one of the four conditions in (\ref{eq-csdf2}), then deciding $(\eta_k,\mu)\in K$ can be reduced to $\eta_{k-2}$ or $\eta_{k-1}$. More precisely, if
$A_1\geq 2k-1$ or $B_1\geq 2k-1$, then from the proof of Proposition \ref{prp:csdnu1} we know that $(\eta_{k-1},\mu-(2k-1))\in K$ or $(\eta_{k-1},\mu'-(2k-1))\in K$
implies $(\eta_{k},\mu)\in K$.
Suppose that $A_2\geq 4k-4$ and let $\nu=\mu-(2k-2,2k-2)\vdash (k-2)^2-1$. If $k$ is even, by Lemma \ref{lem:k22k} we have
$$\big((k,k),(k,k)\big)~\text{and}~ \big((k-2,k-2),(k-2,k-2)\big)\in K.$$
Since $(k-2,k-2)'=(2^{k-2})$, we have $$((2^{k-2}),(k-2,k-2))\in K.$$
Observe that
$$\eta_{k}=\left(\eta_{k-2}+(2^{k-2})\right)
\cup (k,k).$$
If we let $(\eta_{k-2},\nu)\in K$, then
$$(\eta_{k-2}+(2^{k-2}),\nu+(k-2,k-2))\in K,$$
and therefore
$$\left(\left(\eta_{k-2}+(2^{k-2})\right)\cup(k,k),
\nu+(k-2,k-2)+(k,k)\right)=(\eta_{k},\mu)\in K.$$
Similarly, if $B_2\geq 4k-4$, then $(\eta_{k-2},\mu'-(2k-2,2k-2))\in K$ implies $(\eta_{k},\mu)\in K$.
\end{proof}

Suppose that $k$ is odd. By Lemma \ref{lem:k22k} we have $g((k,k),(k,k),(k,k))=0$. So
the argument in Lemma \ref{lem:df2cse} is not suitable for the odd number. Interestingly, by Lemma \ref{lem:k22k} we have $((k^4),(2k,2k))\in K$ when $k$ is odd. By similar argument in Lemma \ref{lem:df2cse}, in the following lemma we will show that if $k\geq 19$ is odd, then deciding $(\eta_k,\mu)\in K$ can be reduced to $\eta_{k-4}$ or $\eta_{k-1}$.
Since the discussions are similar, we just outline the proof here.

\begin{lemma}\label{lem:df2cso}
Suppose that  $k\geq 19$ is odd and $\mu \in D(k^2-1,2)$.
Then deciding $(\eta_k,\mu)\in K$ can be reduced to $\eta_{k-1}$ or $\eta_{k-4}$.
\end{lemma}
\begin{proof}
If $k\geq 19$, then one of the following conditions must occur: \begin{align}\label{eq-csdf2odd}
A_2\geq 8k-16,~~B_2\geq 8k-16,~~A_1\geq 2k-1~\text{and}~B_1\geq 2k-1.
\end{align}
Otherwise, we have
\begin{align*}
 k^2-1=& 4+A_1+A_2+B_1+B_2\\
 \leq& 4+2k-2+8k-17+2k-2\\
 =&20k-34,
 \end{align*}
which contradicts $k\geq19$.

If $\mu \in D(k^2-1,2)$ satisfies one of the four conditions in (\ref{eq-csdf2odd}), then deciding $(\eta_k,\mu)\in K$ can be reduced to $\eta_{k-4}$ or $\eta_{k-1}$.
We just discuss the condition when $A_2\geq 8k-16$ here.

Suppose that $A_2\geq 8k-16$ and let $\nu=\mu-(4k-8,4k-8)\vdash (k-4)^2-1$. If $k$ is odd, by Lemma \ref{lem:k22k} we have $$\left((k^4),(2k,2k)\right)\in K~\text{and}~\left(\left((k-4)^4\right),(2k-8,2k-8)\right)\in K$$
and  from $\left((k-4)^4\right)'=\left(4^{k-4}\right)$
we have
$$\left((4^{k-4}),(2k-8,2k-8)\right)\in K.$$
Observe that $$\eta_{k}=\left(\eta_{k-4}+(4^{k-4})\right)\cup(k^4).$$
So if we let $(\eta_{k-4},\nu)\in K$, then
$(\eta_{k},\mu)\in K$.
\end{proof}
Thus, by Lemma \ref{lem:df2cse} and \ref{lem:df2cso}, for $\mu \in D(k^2-1,2)$ we have the following proposition.
\begin{proposition}
Suppose that  $(\eta_k,\mu)\in K$ for all $k$ such that $1\leq k \leq 18$ and all $\mu \in D(k^2-1,2)$. Then for all $k\geq 19$ and $\mu \in D(k^2-1,2)$, we also have $(\eta_k,\mu)\in K$.
\end{proposition}

\begin{lemma}\label{lem:n3k-1}
For $n\in \mathbb{N}$, if $n$ is odd, we have
$((n,n-2),(n-1,n-1))\in K$.
If $n$ is even, we have $g((n,n-2),(n,n-2),(n-1,n-1))=0$. However, if $n$ is even, we have
$((n,n-2,n-4,n-6),(2n-6,2n-6))\in K$.
\end{lemma}
\begin{proof}
If $n$ is odd, by Theorem 1.7 of \cite{Tewari} we have
$$g((n,n-2),(n,n-2),(n-1,n-1))=g((n,n-2),(n-1,n-1),(n,n-2))=1>0.$$
Moreover, if $n$ is even, we have $g((n,n-2),(n,n-2),(n-1,n-1))=0$.
Since $(n,n-2,n-4,n-6)\unlhd (2n-6,2n-6)$, by Lemma \ref{lem:mmn} we have
$$((n,n-2,n-4,n-6),(2n-6,2n-6))\in K.$$
\end{proof}

\begin{lemma}\label{lem:df2caodd}
Suppose that $\mu \in D(3k^2,2)$, where  $k\geq 12$ and $3k-1$ is odd. Then deciding $(\gamma_k,\mu)\in K$ can be reduced to
$\gamma_{k-1}$ or $\gamma_{k-2}$.

\end{lemma}

\begin{proof}
For $\mu \in D(3k^2,2)$, with notations in Definition \ref{def:alw} we have
$$3k^2=|\mu|=4+A+B=4+A_1+A_2+B_1+B_2.$$
If $k\geq 12$, then one of the following conditions must occur: \begin{align}\label{eq-cardf2}
A_2\geq 12k-12,~~B_2\geq 12k-12,~~A_1\geq 6k-3~\text{and}~B_1\geq 6k-3.     \end{align}
Otherwise, we have
\begin{align*}
 3k^2=& 4+A_1+A_2+B_1+B_2\\
 \leq& 4+12k-13+6k-4+12k-13+6k-4\\
 =&36k-30,
 \end{align*}
which contradicts  $k\geq12$.

Suppose that $3k-1$ is odd.
If $\mu \in D(3k^2,2)$ satisfies one of the four conditions in (\ref{eq-cardf2}), then deciding $(\gamma_k,\mu)\in K$ can be reduced to $\gamma_{k-2}$ or $\gamma_{k-1}$. More precisely, if
$A_1\geq 6k-3$ or $B_1\geq 6k-3$, then from the proof of Proposition \ref{prp:cardnu1} we know that
$(\gamma_{k-1},\mu-(6k-3))\in K$ or $(\gamma_{k-1},\mu'-(6k-3))\in K$ implies $(\gamma_{k},\mu)\in K$.
Suppose that $A_2\geq 12k-12$ and let $\nu=\mu-(6k-6,6k-6)\vdash 3(k-2)^2$. If $3k-1$ is odd, by Lemma \ref{lem:n3k-1} we have $$((3k-1,3k-3),(3k-2,3k-2))\in K~\text{and}~((3k-3,3k-5),(3k-4,3k-4))\in K.$$
and from $(3k-3,3k-5)'=(2^{3k-5},1,1)$ we have
$$((2^{3k-5},1,1),(3k-4,3k-4))\in K.$$
Observe that $$\gamma_{k}=\left(\gamma_{k-2}+(2^{3k-5},1,1)\right)\cup(3k-1,3k-3).$$
If we let $(\gamma_{k-2},\nu)\in K$, then
$$(\gamma_{k-2}+(2^{3k-5},1,1),\nu+(3k-4,3k-4))\in K.$$
Denote $\overline{\gamma_{k-2}}=\gamma_{k-2}+(2^{3k-5},1,1).$ Then we have
$$\left(\overline{\gamma_{k-2}}\cup(3k-1,3k-3),
\nu+(3k-4,3k-4)+(3k-2,3k-2)\right)=(\gamma_{k},\mu)\in K.$$
Similarly, if $B_2\geq 12k-12$, then  $(\gamma_{k-2},\mu'-(6k-6,6k-6))\in K$ implies
$(\gamma_{k},\mu)\in K$.
\end{proof}

Suppose that $3k-1$ is even. By Lemma \ref{lem:n3k-1} we have $g((3k-1,3k-3),(3k-1,3k-3),(3k-2,3k-2))=0$. So
the argument in Lemma \ref{lem:df2caodd} is not suitable for the even number $3k-1$. However, we have $((3k-1,3k-3,3k-5,3k-7),(6k-8,6k-8))\in K$.
So we can also obtain result that is similar to Lemma \ref{lem:df2caodd}. Since the discussions are similar, we just outline the proof here.
\begin{lemma}\label{lem:df2caeven}
Suppose that $\mu \in D(3k^2,2)$, where  $k\geq 19$ and $3k-1$ is even. Then deciding $(\gamma_k,\mu)\in K$ can be reduced to $\gamma_{k-1}$ or $\gamma_{k-4}$.
%
\end{lemma}
\begin{proof}
If $k\geq 19$, then one of the following conditions must occur: \begin{align}\label{eq-cardf2even}
A_2\geq 24k-48,~~B_2\geq 24k-48,~~A_1\geq 6k-3~\text{and}~B_1\geq 6k-3.     \end{align}
Otherwise, we have
\begin{align*}
 3k^2=& 4+A_1+A_2+B_1+B_2\\
 \leq&4+24k-49+6k-4+24k-49+6k-4\\
 =&60k-102,
\end{align*}
which contradicts $k\geq19$.

Suppose that $3k-1$ is even.
If $\mu \in D(3k^2,2)$ satisfies one of the four conditions in (\ref{eq-cardf2even}), then deciding $(\gamma_k,\mu)\in K$ can be reduced to $\gamma_{k-4}$ or $\gamma_{k-1}$.
We just discuss the condition when $A_2\geq 24k-48$.
Suppose that $A_2\geq 24k-48$ and let $\nu=\mu-(12k-24,12k-24)\vdash 3(k-4)^2$. If $3k-1$ is even, by Lemma \ref{lem:n3k-1} we have $$((3k-1,3k-3,3k-5,3k-7),(6k-8,6k-8))\in K,$$
and
$$((3k-5,3k-7,3k-9,3k-11),(6k-16,6k-16))\in K.$$
So by  $(3k-5,3k-7,3k-9,3k-11)'=(4^{3k-11},3,3,2,2,1,1)$, we have
$$\left((4^{3k-11},3,3,2,2,1,1),(6k-16,6k-16)\right)\in K.$$
Observe that $$\gamma_{k}=\left(\gamma_{k-4}+(4^{3k-11},3,3,2,2,1,1)\right)\cup
(3k-1,3k-3,3k-5,3k-7).$$
So if we let $(\gamma_{k-4},\nu)\in K$, then $(\gamma_{k},\mu)\in K$.
%
\end{proof}

Thus, by Lemma \ref{lem:df2caodd} and \ref{lem:df2caeven}, for $\mu \in D(3k^2,2)$ we have the following proposition.
\begin{proposition}
Suppose that  $(\gamma_k,\mu)\in K$ for all $k$ such that $1\leq k \leq 18$ and all $\mu \in D(3k^2,2)$. Then for all $k\geq 19$ and $\mu \in D(3k^2,2)$, we also have $(\gamma_k,\mu)\in K$.
\end{proposition}

\section{Final remarks and problems}\label{se:rem}

\subsection{The distribution of Durfee sizes}
In  \cite{can}, the authors found an asymptotic formula for $|D(n,k)|$. By Corollary 1 there, we can see that for each fixed $k$ the proportion $|D(n,k)|/|P(n)|$ tends to zero if $n\to \infty$. Moreover, the authors  showed that the sequence $\{|D(n,k)|\}$, $0\leq k\leq \lfloor\sqrt{n}\rfloor$, is asymptotically normal, unimodal, and log concave.  The most likely size of the Durfee square  for a partition in $P(n)$ is asymptotic to $(\sqrt{6} \log2/\pi)(\sqrt{n})$.

\subsection{Relation with Permutohedron}
By definition $\Lambda(\mu)=\{\nu|~\nu\unlhd\mu,~\nu\in~P(n)\}$.
In literature  (see e.g. \cite[Sec. 3.1]{stanley}), $\Lambda(\mu)$ (resp. $V(\mu)$) is called  the \emph{principal order ideal} generated by $\mu$ (resp. \emph{principal dual order ideal} generated by $\mu$).
Suppose that $\mu=(\mu_1,\mu_2,\ldots,\mu_n)\vdash n$. If $\ell(\mu)<n$, we let $\mu_i=0$ for $i>\ell(\mu)$.
For $k\geq \ell(\mu)$, $\mu$ can be viewed as a vector in $\mathbb{R}^{k}$.
For each $k\geq \ell(\mu)$, define the {\it $k$-th permutohedron}  $P_k(\mu)\subseteq\mathbb{R}^{k}$ by
\begin{align*}
P_k(\mu):=\text{ConvexHull}\{(\mu_{\sigma(1)},\mu_{\sigma(2)},\ldots,\mu_{\sigma (k)})|~\sigma\in S_k \}.
\end{align*}
It is well known that $\Lambda(\mu)\subseteq P_n(\mu)$, the {\it $n$-th permutohedron}  $P_n(\mu)$ associated to $\mu$ (see e. g. \cite{CaLiu,Postni}).

By the results in \cite{Postni}, can we give an estimate of
$\Lambda(\mu)$? It is nontrival, since when $\mu=(n)$, we have $\Lambda((n))=|P(n)|$ whose asymptotic estimation was given by Hardy and Ramanujan \cite{Andr}.
$P_{m}(\rho_{m-1})$ is called {\it the regular permutohedron} in \cite{Postni}. It was shown that $P_{m}(\rho_{m-1})$ is a {\it graphical zonotope}. It would be interesting to decide whether $P_{\frac{m(m-1)}{2}}(\rho_{m-1})$ is a  graphical zonotope or not.

\subsection{Irreducible characters vanishing on $\rho_m$ and $\widehat{\rho_m}$}

If $\chi^\lambda(\rho_m)$ and $\chi^\lambda(\widehat{\rho_m})=0$, by Pak and Bessenrodt's criteria we can't decide whether $[\lambda]$ appears in $[\rho_m]\otimes[\rho_m]$ or not \cite{Bessenrodt17,PPV}. Thus
Pak and Bessenrodt's criteria lead us to study the non-zero character values on a fixed conjugacy class which correspond to the non-zero elements on columns of the character table. It can also be seen to connect to work on nonvanishing conjugacy classes, that is, conjugacy classes on which no irreducible character vanishes. A partition is a nonvanishing partition if it labels a nonvanishing conjugacy class of a symmetric group.

Let  $N(\mu)$ denote the number of irreducible characters $\chi^\lambda$ such that $\chi^\lambda(\mu)=0$.
It has been shown that any nonvanishing partition should be of the form $(3^a, 2^b, 1^c)$ for some $a, b, c\geq0$ \cite{GLe,Morotti}.  Thus  we should have
$N(\rho_m)$, $N(\widehat{\rho_m})>0$ for $m\geq4$.
By  known results (see e. g. Remark 4.5 in \cite{Bessenrodt17}), we can see that irreducible characters vanishing on $\widehat{\rho_m}$ may be more than $\rho_m$. Thus we have the following problem.
\begin{problem}
Does it  hold that $N(\rho_m)<N(\widehat{\rho_m})$ for  $m\geq3$?
\end{problem}

The following problem gives an upper bound of non-zero elements on columns of the character table of $S_n$. It was mentioned in \cite{Morotti} and raised by A. Evseev.
\begin{problem} \cite{Morotti}
Let $\pi\in S_n$. Does it always hold that the number of irreducible characters of $S_n$ not vanishing on $\pi$ is at most equal to the number of irreducible characters of $C_{S_n}(\pi)$?
\end{problem}

\subsection{Some complexity observations}

It was shown in \cite[Thm 7.1]{PP} that for any $\lambda$, $\nu$ deciding whether $\chi^\lambda(\nu)=0$ is NP-hard. It can be reduced to the classical NP-complete Knapsack problem.  Expanding the power sum  function $p_\mu$ into Schur functions $s_\lambda$ (see \cite[Cor. 7.17.4]{stanley})
we have
$p_\mu=\sum_{\lambda\in P(n)}\chi^\lambda(\mu)s_\lambda.$
Hence, the irreducible character values on $\mu$ are the coefficients in the expansion of $p_\mu$. In \cite[Cor. 4.3]{BFom}, the authors showed that  there exist probabilistic polynomial time algorithms for computing an expansion  of a given power sum $p_\mu$. That is, there exist probabilistic polynomial time algorithms for computing the set \{$\chi^\lambda(\mu)\mid\lambda\in P(n)$\}.  Similarly, in \cite{Hep94} it was shown that deciding whether $\chi^\lambda(\mu)=0$ can be done in probabilistic polynomial time.

\subsection{The Staircase-like partition}

For each $n$, we consider self-conjugate partitions that are close to staircase partitions as follows. For each $n\in \mathbb{N}$, there exist $m$, $k$ such that $n=\frac{m(m+1)}{2}+k$ where $0\leq k\leq m$. It is not hard to verify the following conditions.

(5.5.1) If $m$ is even, then for each $k$ there are self-conjugate partitions  $\lambda\vdash n$ such that $\rho_m\subseteq\lambda\subseteq \rho_{m+1}$.

(5.5.2) If $m$ is odd and $k$ is even, then for each $k$ there are self-conjugate partitions  $\lambda\vdash n$ such that $\rho_m\subseteq\lambda\subseteq \rho_{m+1}$.

(5.5.3) If $m$ and $k$ are odd, then no self-conjugate partitions of $n$ lie between $\rho_m$ and $\rho_{m+1}$. But we can find self-conjugate partitions $\lambda\vdash n$ such that $\rho_{m-1}\subseteq\lambda\subseteq \rho_{m+2}$.

 In fact, if $k=1$ we can find self-conjugate partitions $\lambda\vdash n$ such that $\rho_{m-1}\subseteq\lambda\subseteq \rho_{m+1}$.
For example, if $n=7$, then $m=3$ and $k=1$. We let $\lambda=(4,1,1,1)$.
If $k=3,5,7...$ we can find self-conjugate partitions $\lambda\vdash n$ such that $\rho_{m}\subseteq\lambda\subseteq \rho_{m+2}$.
For example, if $n=18$, then $m=5$ and $k=3$. We let $\lambda=(5,4,4,4,1)$.

\begin{definition}
We will call it a \emph{staircase-like partition} if a
self-conjugate partition satisfies one of three conditions in (5.5.1), (5.5.2) and (5.5.3) above.
\end{definition}

There exist  staircase-like partitions for each $n\geq3$. If $n$ is a triangular number, i.e. of the form $m(m+1)/2$, then the corresponding staircase-like partition is just $\rho_m$.
Comparing with Conjecture 1.1 of \cite{PPV}, we propose a generalised Saxl Conjecture as follows.
\begin{conjecture}[\emph{Generalised Saxl Conjecture}]\label{con:slp}
For any $\lambda$ of size $n$ different from 2, 4, 9, if $\lambda$ is a staircase-like partition, then $[\lambda]\otimes[\lambda]$ contains all irreducible representations of $S_n$ as constituents.
\end{conjecture}
With the help of computer, we can easily verify Conjecture \ref{con:slp} for $n\leq 35$.
It is interesting to see that Bessenrodt et al. also proposed a generalised Saxl Conjecture which is related to $p$-cores \cite{Bessenrodt19}.

If $\lambda$ is  self-conjugate, then in several ways we can add 1 or 2 boxes on $\lambda$ to make it become another self-conjugate partition.  For example, adding a box on $\rho_4=(4,3,2,1)$ we get $(4,3,3,1)$ which is self-conjugate.  Adding 2 boxes on $\rho_4=(4,3,2,1)$ we get a self-conjugate partition $(5,3,2,1,1)$. We would like to know the growth behavior of Kronecker coefficient as the growth of partitions.
We raise the following problem. Related discussions can be found in \cite{BriRR}.

\begin{problem}\label{prb:add}
For $\lambda\vdash n$,  suppose that $[\lambda]\otimes[\lambda]$ contains all irreducible representations of $S_n$  as constituents. By adding at most 2 boxes on $\lambda$ we get another self-conjugate partition $\mu$ (not uniquely) such that $\lambda\subseteq\mu$ and $|\mu\setminus\lambda|=$2 (or 1).  Does there always exist some $\mu$ such that $[\mu]\otimes[\mu]$ also contains all irreducible representations of $S_{n+2}$ (or $S_{n+1}$)  as  constituents?
\end{problem}

If Conjecture \ref{con:slp} is true, then it can be viewed as a special case of Problem \ref{prb:add}. In fact, we can get $\rho_m$ by adding 1 or 2 boxes on $\rho_{m-1}$ step by step. In each step, we add 1 or 2 boxes on $\rho_{m-1}$ such that the partition is staircase-like.
For example, when $m=5$, one typical process can be illustrated by Young diagrams as follows.
{\tiny\begin{equation*}
\Yvcentermath1
\yng(4,3,2,1)\longrightarrow\yng(5,3,2,1,1)\longrightarrow\yng(5,4,2,2,1)
\longrightarrow\yng(5,4,3,2,1).
\normalsize
\end{equation*}}


\section*{Acknowledgments}
We are very grateful to the editors and referees for their valuable comments and suggestions. Thanks to my undergraduate students Zhenyu Luo and Yuan Jiang for their helpful discussions when I prepare this paper. Many thanks to Haiyan Gao, Xiaorui Hu, Wan Wan, Yaning Xie, Bing Yang and Yanfang Zhang for their helpful advices. Special thanks to Professor Fu Liu for her helpful comments. We also thank Professor John Stembridge for making his Maple package SF \cite{stem} freely available.

\bibliographystyle{amsplain}

\end{document}